\documentclass[%
 aapm,
 mph,%
 amsmath,amssymb,
 reprint,%
]{revtex4-1}
\usepackage{graphicx}% Include figure files
\usepackage{dcolumn}% Align table columns on decimal point
\usepackage[mathlines]{lineno}% Enable numbering of text and display math
\draft % marks overfull lines with a black rule on the right
\usepackage{amsmath}
\usepackage{amsfonts}
\usepackage{amssymb}
\usepackage{graphicx, times}
\usepackage{subfig}
\usepackage{amsthm}
\usepackage{bm}
\usepackage[section]{placeins}  
\usepackage[dvipsnames,table]{xcolor}
\usepackage{listings}

 \definecolor{gris245}{RGB}{245,245,245}
 \definecolor{olive}{RGB}{50,140,50}
 \definecolor{brun}{RGB}{175,100,80}
\newtheorem{The}{Theorem}[section]
\newtheorem{Def}{Definition}[section]

\newtheorem{Ex}{Example}[section]
\newtheorem{remark}{Remark}[section]

\newtheorem{Conj}{Conjecture}
\begin{document}

% Use the \preprint command to place your local institutional report number
% on the title page in preprint mode.
% Multiple \preprint commands are allowed.
%\preprint{}

\title{Singular points in the solution trajectories of fractional order dynamical systems} %Title of paper

% repeat the \author .. \affiliation  etc. as needed
% \email, \thanks, \homepage, \altaffiliation all apply to the current author.
% Explanatory text should go in the []'s,
% actual e-mail address or url should go in the {}'s for \email and \homepage.
% Please use the appropriate macro for the type of information

% \affiliation command applies to all authors since the last \affiliation command.
% The \affiliation command should follow the other information.

\author{Sachin Bhalekar}
\email{sbb\_maths@unishivaji.ac.in, sachin.math@yahoo.co.in}
%\homepage[]{Your web page}
%\thanks{}
%\altaffiliation{}
\affiliation{Department of Mathematics, Shivaji University, Kolhapur - 416004, India}
\author{Madhuri Patil}
\email{madhuripatil4246@gmail.com}
%\homepage[]{Your web page}
%\thanks{}
%\altaffiliation{}
\affiliation{Department of Mathematics, Shivaji University, Kolhapur - 416004, India}
\date{\today}

\begin{abstract}
Dynamical systems involving non-local derivative operators are of great importance in Mathematical analysis and applications. This article deals with the dynamics of fractional order systems involving Caputo derivatives. We take a review of the solutions of linear dynamical systems ${}_0^C\mathrm{D}_t^\alpha X(t)=AX(t)$, where the coefficient matrix $A$ is in canonical form. We describe exact solutions for all the cases of canonical forms and sketch phase portraits of planar systems.
\par We discuss the behavior of the trajectories when the eigenvalues $\lambda$ of $A$ are at the boundary of stable region i.e. $|arg(\lambda)|=\frac{\alpha\pi}{2}$. Further, we discuss the existence of singular points in the trajectories of such systems in a region of $\mathbb{C}$ viz. Region II. It is conjectured that there exists singular point in the solution trajectories if and only if $\lambda\in$ Region II.
\end{abstract}

%\pacs{}% insert suggested PACS numbers in braces on next line

\maketitle %\maketitle must follow title, authors, abstract and \pacs

\begin{quotation}
%The ``lead paragraph'' is encapsulated with the \LaTeX\
%\verb+quotation+ environment and is formatted as a single paragraph before the first section heading.
%(The \verb+quotation+ environment reverts to its usual meaning after the first sectioning command.)
%Note that numbered references are allowed in the lead paragraph.
%
%The lead paragraph will only be found in an article being prepared for the journal \textit{Chaos}.
The systems involving nonlocal operators are proved useful in modeling natural phenomena. In contrast with classical operators, these are able to model memory in the system. However, the behavior of non-local models may differ from those containing integer-order derivatives with respect to some aspects. This article focuses one of these aspects which is important in chaos theory viz. self-intersecting trajectories.
\end{quotation}

\section{Introduction}
Nonlocal operators play a vital role in Mathematical analysis and applications. If the order of the derivative involved in the system is a non-integer then it is called as fractional derivative. In contrast with classical (integer-order) derivative,  fractional derivatives are non-local. There are different approaches of defining fractional order derivatives \cite{Das, Katugampola, Atangana, Hristov, Podlubny}. All these approaches have their own importance. 
\par Basic theory and applications of fractional calculus is presented in books \cite{Podlubny, Oldhalm, Samko}. Special functions arising in fractional calculus and their applications can be found in books \cite{Kiryakova, Mathai H, Mathai Sxena H}. Applications of fractional calculus in linear viscoelasticity \cite{Mainardi} and in bioengineering \cite{Magin} are discussed by Mainardi and Magin respectively. Analysis of fractional differential equations is carried out by various researchers e.g. see \cite{Diethelm Ford, Babakhani}.
\par The detailed review and history of fractional calculus is presented in the paper \cite{Machado} and two nice posters \cite{Old history,Recent history} designed by Machado, Kiryakova and Mainardi. Stability results for fractional differential equations are initially proposed by Matignon \cite{Matignon1, Matignon2, Matignon3, Matignon4} and later developed by Deng \cite{Deng}, Tavazoei and Haeri \cite{Tavazoei} and so on. Stability of nonlinear delay differential equations is discussed by Bhalekar in \cite{Bhalekar}. Chaos in fractional order nonlinear systems is discussed in \cite{v Gejji,Kaslik Shiv,Hartley,C. Li,Zhang}.
\par In this article we take a review of stability results in linear time invariant fractional order systems involving Caputo derivatives. We discuss the behavior of the system when the eigenvalues of coefficient matrix are on the boundary of stable region. Further, we discuss the singular points such as cusp and multiple-points in the solution trajectories. 

\section{Preliminaries}
This section deals with basic definitions and results given in the literature \cite{Das, Podlubny, Samko, Erdelyi, Diethelm}. Throughout this section, we take $n\in\mathbb{N}$.
\begin{Def}
The Riemann-Liouville (\text RL) fractional integral is defined as,
\begin{multline}
{}_0\mathrm{I}_t^\alpha f(t)=
\frac{1}{\Gamma{(\alpha)}}\int_0^t (t-\tau)^{\alpha-1}f(\tau)\,\mathrm{d}\tau,
  \\ \mathrm{if}\,\, n-1<\alpha< n.
\end{multline}
\end{Def}
\begin{Def}
The Riemann-Liouville (\text RL) fractional derivative is defined as, 
\begin{equation}
{}_0^{RL}\mathrm{D}_t^\alpha f(t)=
\begin{cases}
\frac{1}{\Gamma{(n-\alpha)}}\frac{d^n}{dt^n}\int_0^t (t-\tau)^{n-\alpha-1}f(\tau)\,\mathrm{d}\tau, &\\ \qquad \qquad \qquad \mathrm{if}\,\, n-1<\alpha< n\\
\frac{d^n}{dt^n}f(t), \qquad \qquad \mathrm{if}\, \alpha=n.
\end{cases}
\end{equation}
\end{Def}
\begin{Def}
The Caputo fractional derivative is defined as,
\begin{equation}
{}_0^{C}\mathrm{D}_t^\alpha f(t)=
\begin{cases}
\frac{1}{\Gamma{(n-\alpha)}}\int_0^t (t-\tau)^{n-\alpha-1}f^{(n)}(\tau)\,\mathrm{d}\tau, & \\ \qquad \qquad \qquad \mathrm{if}\,\, n-1<\alpha< n\\
\frac{d^n}{dt^n}f(t), \qquad \qquad \mathrm{if}\,\, \alpha=n.
\end{cases}
\end{equation}
Note that ${}_0^{C}\mathrm{D}_t^\alpha c=0$, where $c$ is constant.
\end{Def}
\begin{Def}
The one-parameter Mittag-Leffler function is defined as,
\begin{equation}
E_\alpha(z)=\sum_{k=0}^\infty \frac{z^k}{\Gamma(\alpha k+1)}\, ,\qquad (\alpha>0).
\end{equation}
The two-parameter Mittag-Leffler function is defined as,
\begin{equation}
E_{\alpha,\beta}(z)=\sum_{k=0}^\infty \frac{z^k}{\Gamma(\alpha k+\beta)}\, ,\qquad (\alpha>0,\,\beta>0).
\end{equation}
\end{Def}
\begin{Def}
A steady state solution of 
\begin{equation}
{}_0^C\mathrm{D}_t^\alpha x=f(x)
\end{equation} 
is called an equilibrium point. Thus $x_*$ is an equilibrium point of (6) if $f(x_*)=0$.
\end{Def}

Note that, if the Caputo derivative in (6) is replaced by RL derivative then the only equilibrium points are $x_*=0$ (if exists).

\noindent{\bf Properties:}\cite{Podlubny} 
If $\mathcal{L}\left\{f(t);s\right\}=F(s)$ is the Laplace transform of $f$ then we have:
\begin{equation}
\mathcal{L}\left\{{}_0^{RL}\mathrm{D}_t^\alpha f(t);s\right\}=s^\alpha F(s)-\sum_{k=0}^{n-1} s^k [{}_0^{RL}\mathrm{D}_t^{\alpha-k-1} f(t)]_{t=0},  
\end{equation}
where\, $n-1\le\alpha< n.$
\begin{equation}
\mathcal{L}\left\{{}_0^C\mathrm{D}_t^\alpha f(t);s\right\}=s^\alpha F(s)-\sum_{k=0}^{n-1} s^{\alpha-k-1} f^{(k)}(0),
\end{equation}
where\, $n-1<\alpha\le n.$
\begin{equation}
\mathcal{L}\left\{t^{\alpha k+\beta-1} E_{\alpha,\beta}^{(k)}(\pm\lambda t^\alpha)\right\}=\frac{k! s^{\alpha-\beta}}{(s^\alpha\mp\lambda)^{k+1}},
\end{equation}
where\,\, $\alpha>0, \beta>0$, $k\in \mathbb{N}\cup \{0\}$, $E_{\alpha,\beta}^{(k)}(y)=\frac{d^k}{dy^k}E_{\alpha,\beta}(y)$ and  $Re[s]>|\lambda|^{1/\alpha}$.
\\
\begin{Def}
 A function $f:(0,\infty)\rightarrow\mathbb{R}$ is called completely monotonic if it possesses derivatives $f^{(n)}(x)$ of any order $n = 0, 1, 2, \dots$, and the derivatives are alternating in sign\\
 i.e. $$
 (-1)^nf^{(n)}(x)\ge 0,\qquad \forall x\in(0,\infty).
 $$
\end{Def}
\begin{The} 
\cite{Luchko} Solution of homogeneous fractional order differential equation
\begin{equation}
{}_0^C\mathrm{D}_t^\alpha y(x)+\lambda y(x)=0, \qquad 0<\alpha<1
\end{equation} 
is given by,
\begin{equation}
y(x)=y(0)E_\alpha(-\lambda x^\alpha).
\end{equation}
\end{The}
\begin{The}
\cite{Luchko} Solution of non-homogeneous fractional order differential equation
\begin{equation}
{}_0^C\mathrm{D}_t^\alpha y(x)+\lambda y(x)=g(x), \qquad 0<\alpha<1
\end{equation} 
is given by,
\begin{equation}
y(x)=\int_0^x t^{\alpha-1}E_{\alpha,\alpha}(-\lambda t^\alpha)g(x-t)\,\mathrm{d}t+y(0)E_\alpha(-\lambda x^\alpha).
\end{equation}
\end{The}
\begin{The} \cite{Podlubny}
If $0<\alpha<2$, $\beta$ is an arbitrary complex number and $\mu$ is an arbitrary real number such that $\frac{\alpha\pi}{2}<\mu<min\{\pi,\pi\alpha\}$, then for an arbitrary integer $p\ge1$ the following expansion holds:
\begin{equation}
\begin{split}
E_{\alpha,\beta}(re^{i\theta}t^\alpha) & = \frac{1}{\alpha}(re^{i\theta}t^\alpha)^{\frac{1-\beta}{\alpha}}exp((re^{i\theta}t^\alpha)^{\frac{1}{\alpha}})\\
& -\sum_{k=1}^p \frac{(re^{i\theta}t^\alpha)^{-k}}{\Gamma(\beta-\alpha k)}+O(|re^{i\theta}t^\alpha|^{-1-p}),
\end{split}
\end{equation} 
$|e^{i\theta}t^\alpha|\rightarrow\infty$,\, $|arg(e^{i\theta}t^\alpha)|=|\theta|\le\mu$.
\end{The}

\section{Dynamics of fractional order systems}\label{Sec3}
\subsection{Linear systems involving Caputo fractional derivatives}
Consider the system,\\
\begin{equation}
{}_0^C\mathrm{D}_t^\alpha X(t)=AX(t),\qquad X(0)=X_0, \qquad 0<\alpha<1
\end{equation}
where $X$ is column vector in $\mathbb{R}^n$, $A$ is $n\times n$ matrix and constant vector $X_0=[C_1,C_2,\dots C_n]^T$.
We analyze this system by using the same approach as in \cite{Hirsch}. Another approach can be found in \cite{Odibat}.
Without loss of generality, we can assume $A$ is in canonical form and we discuss solution of the system (15) in the following cases:\\ 
\\
{\bf Case I :-} 
Suppose that 
\begin{equation*}
A=\begin{bmatrix}
\lambda_1 & 0 & \dots & 0\\
0 & \lambda_2 & \dots & 0\\
\hdotsfor{4}\\
 0 & 0 & \dots & \lambda_n 
\end{bmatrix}
\end{equation*}
where $\lambda_1,\lambda_2,\dots,\lambda_n$ are real and distinct.
Then the solution of the system (15) can be given by using (11) as,
\begin{equation*}
X(t)=
\begin{bmatrix}
E_\alpha(\lambda_1t^\alpha) & 0 & \dots & 0\\
0 & E_\alpha(\lambda_2t^\alpha) & \dots & 0\\
\hdotsfor{4}\\
 0 & 0 & \dots & E_\alpha(\lambda_nt^\alpha)
 \end{bmatrix}
 \begin{bmatrix}
C_1\\
C_2\\
\vdots\\
C_n
 \end{bmatrix}.
\end{equation*}\\
{\bf Case II :-} Suppose that the eigenvalues of $A$ are real and repeated. If $\lambda$ is repeated eigenvalue of $A$ and the geometric multiplicity of $\lambda$ is same as its algebraic multiplicity, then the solution is obtained in a same way as in Case I.\\ 
The cases of geometric multiplicity $<$ algebraic multiplicity with $n=3$ are illustrated below:\\
(i) Suppose that, 
$
 A=\begin{bmatrix}
\lambda & 1 & 0\\
0 & \lambda & 0\\
 0 & 0 & \lambda 
\end{bmatrix}
$.\\
The second and third equations in (15)  are decoupled in this case.\\
Therefore, $x_2(t)=E_\alpha(\lambda t^\alpha)C_2$ and  
$x_3(t)=E_\alpha(\lambda t^\alpha)C_3.$\\
Now, the first equation in (15) becomes, 
$$
{}_0^C\mathrm{D}_t^\alpha x_1(t)=\lambda x_1+x_2.
$$
 \begin{equation*}
 \therefore x_1(t)=E_\alpha(\lambda t^\alpha)C_1+t^\alpha E_\alpha^{(1)}(\lambda t^\alpha)C_2.
 \end{equation*}
\begin{equation*}
\therefore X(t)=
\begin{bmatrix}
E_\alpha(\lambda t^\alpha) & t^\alpha E_\alpha^{(1)}(\lambda t^\alpha) & 0\\
0 & E_\alpha(\lambda t^\alpha) & 0\\
 0 & 0 & E_\alpha(\lambda t^\alpha)
 \end{bmatrix}
 \begin{bmatrix}
C_1\\
C_2\\
C_3
 \end{bmatrix}.
\end{equation*} 
(ii) Suppose
$
 A=\begin{bmatrix}
\lambda & 1 & 0\\
0 & \lambda & 1\\
 0 & 0 & \lambda 
\end{bmatrix}
$.\\
Using similar arguments,
\begin{equation*}
X(t)=
\begin{bmatrix}
E_\alpha(\lambda t^\alpha) & t^\alpha E_\alpha^{(1)}(\lambda t^\alpha) & t^{2\alpha} E_\alpha^{(2)}(\lambda t^\alpha)\\
0 & E_\alpha(\lambda t^\alpha) & t^\alpha E_\alpha^{(1)}(\lambda t^\alpha)\\
 0 & 0 & E_\alpha(\lambda t^\alpha)
 \end{bmatrix}
 \begin{bmatrix}
C_1\\
C_2\\
C_3
 \end{bmatrix}.
\end{equation*}
The higher order systems can be solved in a similar way.\\
{\bf Case III :-}
If $A=\begin{bmatrix}
a & b\\
-b & a
\end{bmatrix}$, where $b\ne0$ then $\lambda=a\pm ib=re^{i\theta}$ are complex conjugate eigenvalues of $A$. In this case the system (15) is equivalent to,
\begin{align}
{}_0^C\mathrm{D}_t^\alpha x & = ax+by, \nonumber\\
{}_0^C\mathrm{D}_t^\alpha y & = -bx+ay.
\end{align}
Using method of successive approximations, we get the following solution 
\begin{equation}
X(t)=\begin{bmatrix}
 Re[E_\alpha(\lambda t^\alpha)] & Im[E_\alpha(\lambda t^\alpha)]\\
 -Im[E_\alpha(\lambda t^\alpha)] & Re[E_\alpha(\lambda t^\alpha)]
  \end{bmatrix}
  \begin{bmatrix}
  C_1\\
  C_2
  \end{bmatrix}.
\end{equation}
 In general, if 
\begin{equation*}
A=\begin{bmatrix}
D_1 & 0 & \dots & 0\\
0 & D_2 & \dots & 0\\
\hdotsfor{4}\\
 0 & 0 & \dots & D_n 
\end{bmatrix}
\end{equation*}
where,
\begin{equation*}
D_j=
\begin{bmatrix}
a_j & b_j\\
-b_j & a_j
\end{bmatrix}, \qquad j=1, 2,\dots,n
\end{equation*}
then the solution of the system (15) is,
 \begin{equation*}
 X(t)=
\begin{bmatrix}
 B_1 & 0 & \dots & 0 & 0\\
 0 & B_2 & \dots & 0 & 0\\
 \hdotsfor[2]{5}\\
 0 & 0 & \dots & B_{n-1} & 0\\
 0 & 0 & \dots & 0 & B_n
 \end{bmatrix}
 \begin{bmatrix}
 C_1\\
 C_2\\
 \vdots\\
 C_{2n-1}\\
 C_{2n}
 \end{bmatrix}
 \end{equation*}
 where $2\times2$ matrix  
 $$B_k=\begin{bmatrix}
 Re[E_\alpha(\lambda_kt^\alpha)] & Im[E_\alpha(\lambda_kt^\alpha)]\\
 -Im[E_\alpha(\lambda_kt^\alpha)] & Re[E_\alpha(\lambda_kt^\alpha)]
  \end{bmatrix}.$$
  \begin{remark}
  If $A$ is not in standard canonical form then there exists a matrix $P$ such that, $P^{-1}AP=B$, is in standard canonical form.\\
  Now we solve the system,
  $$
  {}_0^C\mathrm{D}_t^\alpha Y(t)=BY(t)=(P^{-1}AP)Y(t)
  $$
  Finally X(t)=PY(t) gives solution of the given system (15).
  \end{remark}
  \subsection{Stability analysis of the systems of linear fractional differential equations}
  \begin{The} \label{Thm3.1}
   \cite{Matignon1} Consider
  \begin{equation}
  {}_0^C\mathrm{D}_t^\alpha X(t)=AX(t),\quad X(0)=X_0, \quad 0<\alpha<1.
  \end{equation}
  The autonomous system (18) is:\\
  (a) Asymptotically stable if and only if $|arg(specA)|>\frac{\alpha \pi}{2}$.\\
  In this case, the components of the state decay towards 0 like $t^{-\alpha}$.\\
  (b)Stable if and only if either it is asymptotically stable, or those critical eigenvalues which satisfy $|arg(specA)|=\frac{\alpha \pi}{2}$ have geometric multiplicity one.
  \end{The}
  \begin{The} 
  \cite{Qian} Consider 
  \begin{equation}
  {}_0^{RL}\mathrm{D}_t^\alpha X(t)=AX(t),\quad X(0)=X_0, \quad 0<\alpha<1.
  \end{equation}
  The autonomus system (19) is:\\
  (a) Asymptotically stable if all the eigenvalues of $A$ satisfy $|arg(\lambda(A))|>\frac{\alpha \pi}{2}$.\\
  In this case, the components of the state decay towards 0 like $t^{-1-\alpha}$.\\
  (b)Stable if and only if either it is asymptotically stable, or those critical eigenvalues which satisfy $|arg(specA)|=\frac{\alpha \pi}{2}$ have the same algebraic and geometric multiplicity.
  \end{The}
  \subsection{Phase portraits}
  Consider the autonomous system (15), with n=2.
  Clearly equilibrium point of the system (15) is origin.
    \subsubsection{ Eigenvalues of $A$ are in the unstable region:-}
    If $|arg(\lambda)|<\frac{\alpha \pi}{2}$ then the eigenvalue $\lambda$ of $A$ is unstable in the view of Theorem \ref{Thm3.1}. We consider following three cases of unstable eigenvalues. \\[0.20cm]
    Case(i): Suppose $\lambda_1$ and $\lambda_2$ are positive, real eigenvalues of $A$.\\
    In this case the solutions of the system (15) tend away from the origin as shown in the Figure (1). Equilibrium point in this case is called as a source \cite{Hirsch}.
    \begin{figure}[h]
          \subfloat[$\alpha=0.1$.]{\includegraphics[width=0.45\textwidth]{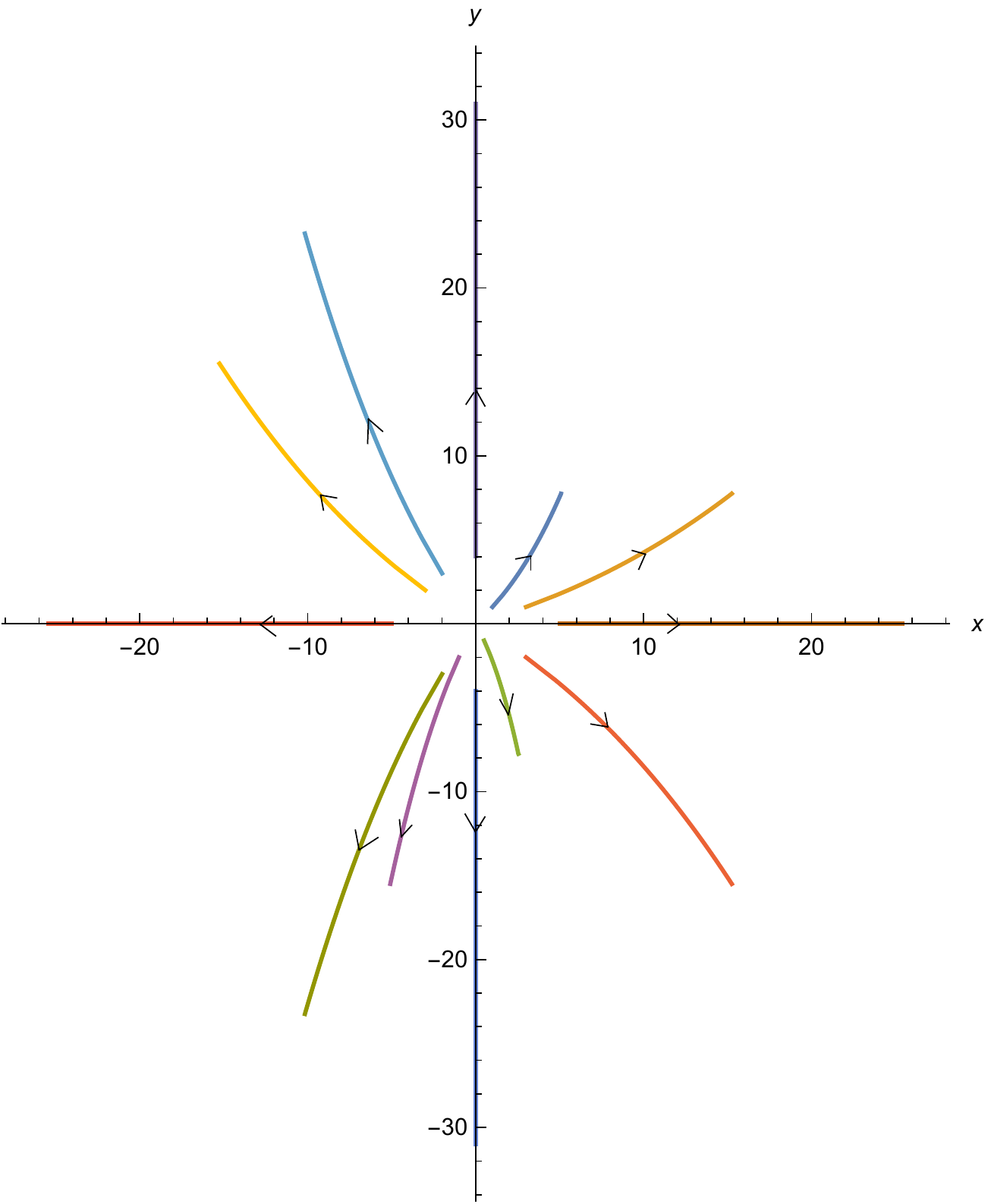}}
          \hfill 
          \subfloat[$\alpha=0.9$]{\includegraphics[width=0.45\textwidth]{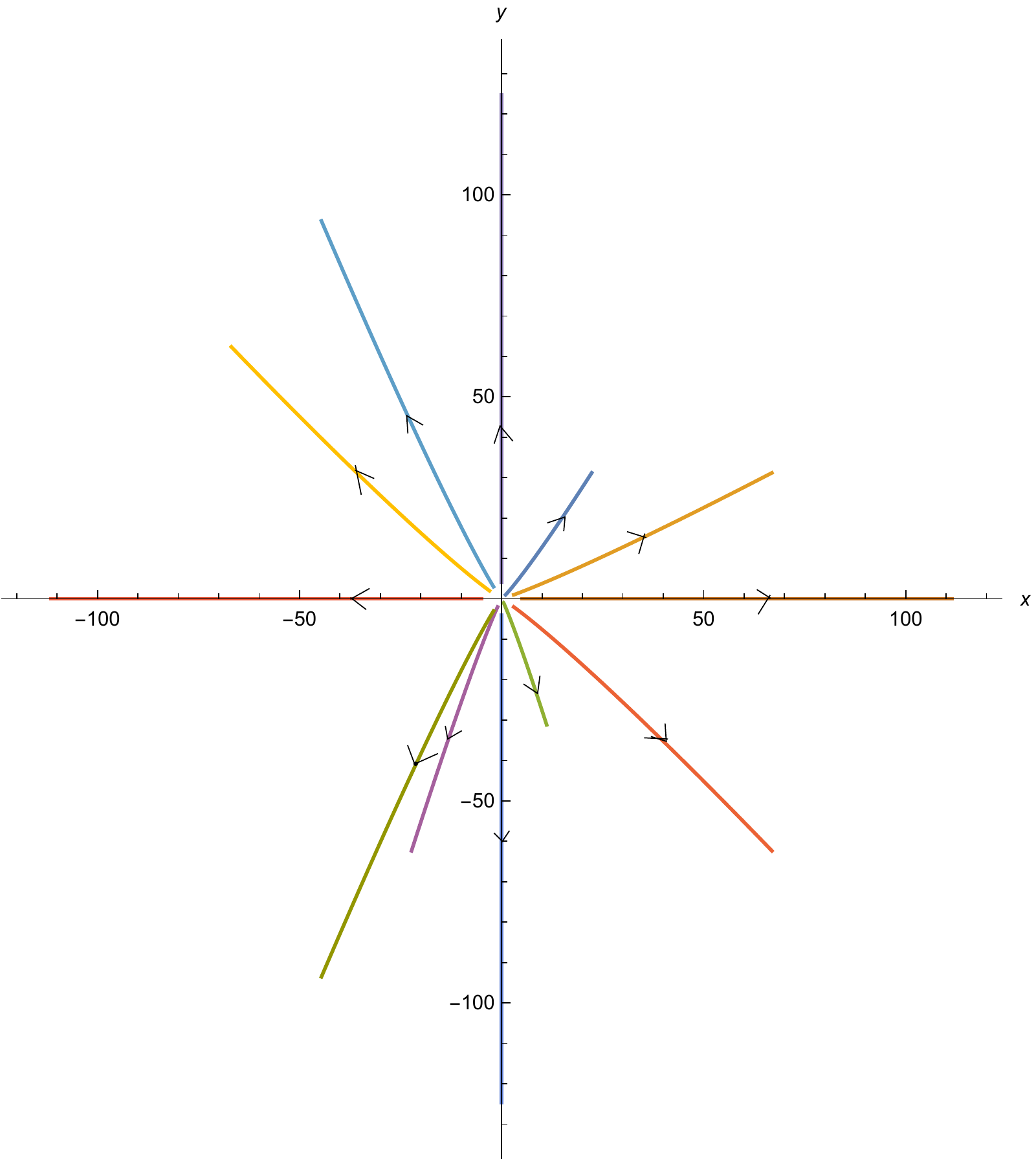}}
          \caption{Source\\ $\lambda_1=1$ and $\lambda_2=1.1$}
        \end{figure}
    \\[0.20cm]
    Case(ii): If 
    $
    A=
    \begin{bmatrix}
    \lambda_1 & 0\\
    0 & \lambda_2
    \end{bmatrix} 
    $
    where $\lambda_2<0<\lambda_1$, then solutions along the $X$-axis moves away from origin and the solutions along the $Y$-axis tends toward origin as $t\rightarrow\infty$. All other solutions tend away from origin in the direction parallel to the $X$-axis as $t\rightarrow\infty$. In backward time, these solutions tend away from origin in the direction parallel to the $Y$-axis as shown in the Figure (2). Equilibrium point in this case is called a saddle \cite{Hirsch}. 
    \begin{figure}[h]
        \subfloat[$\alpha=0.2$.]{\includegraphics[width=0.4\textwidth]{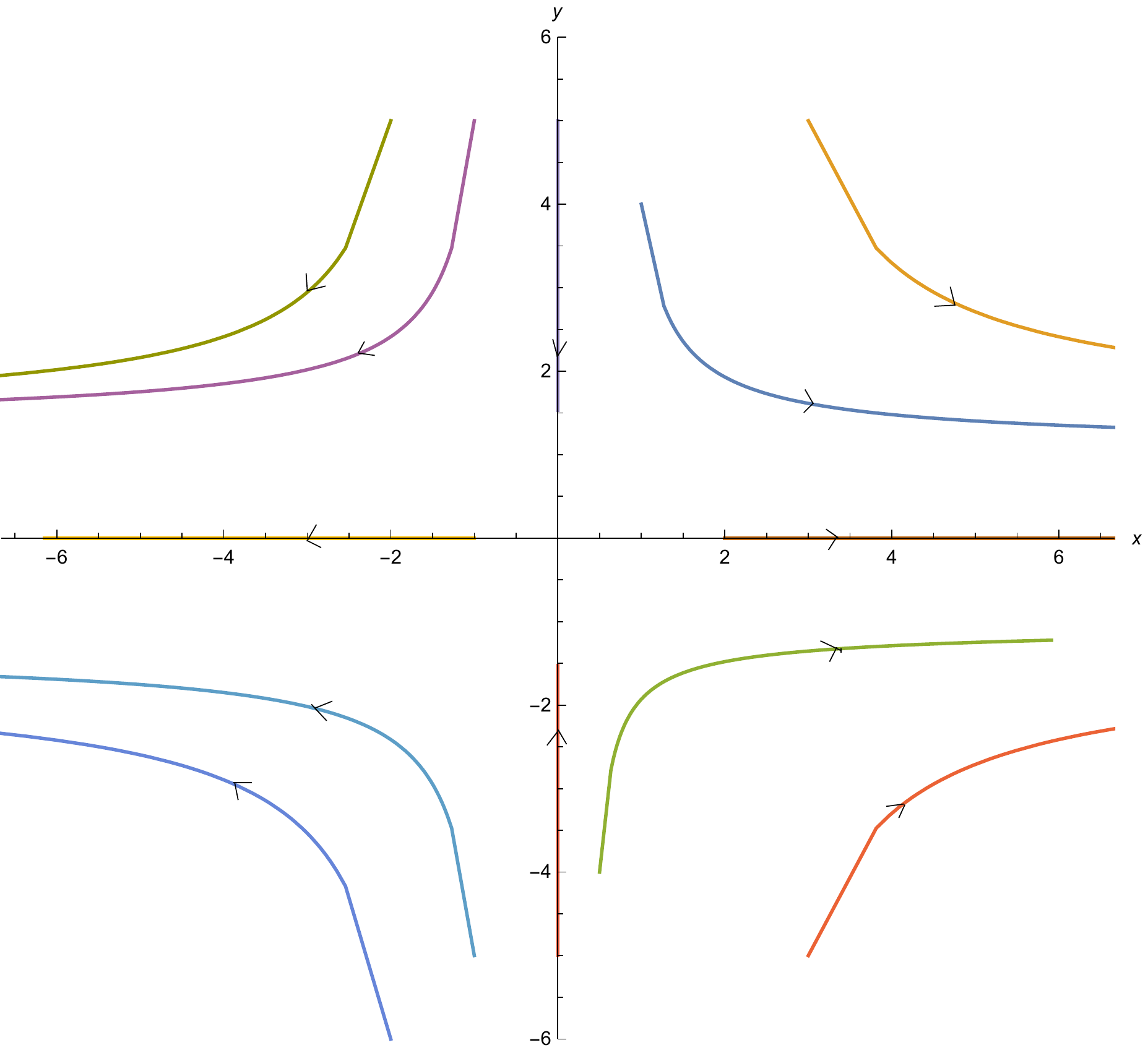}}
        \hfill 
        \subfloat[$\alpha=0.7$]{\includegraphics[width=0.4\textwidth]{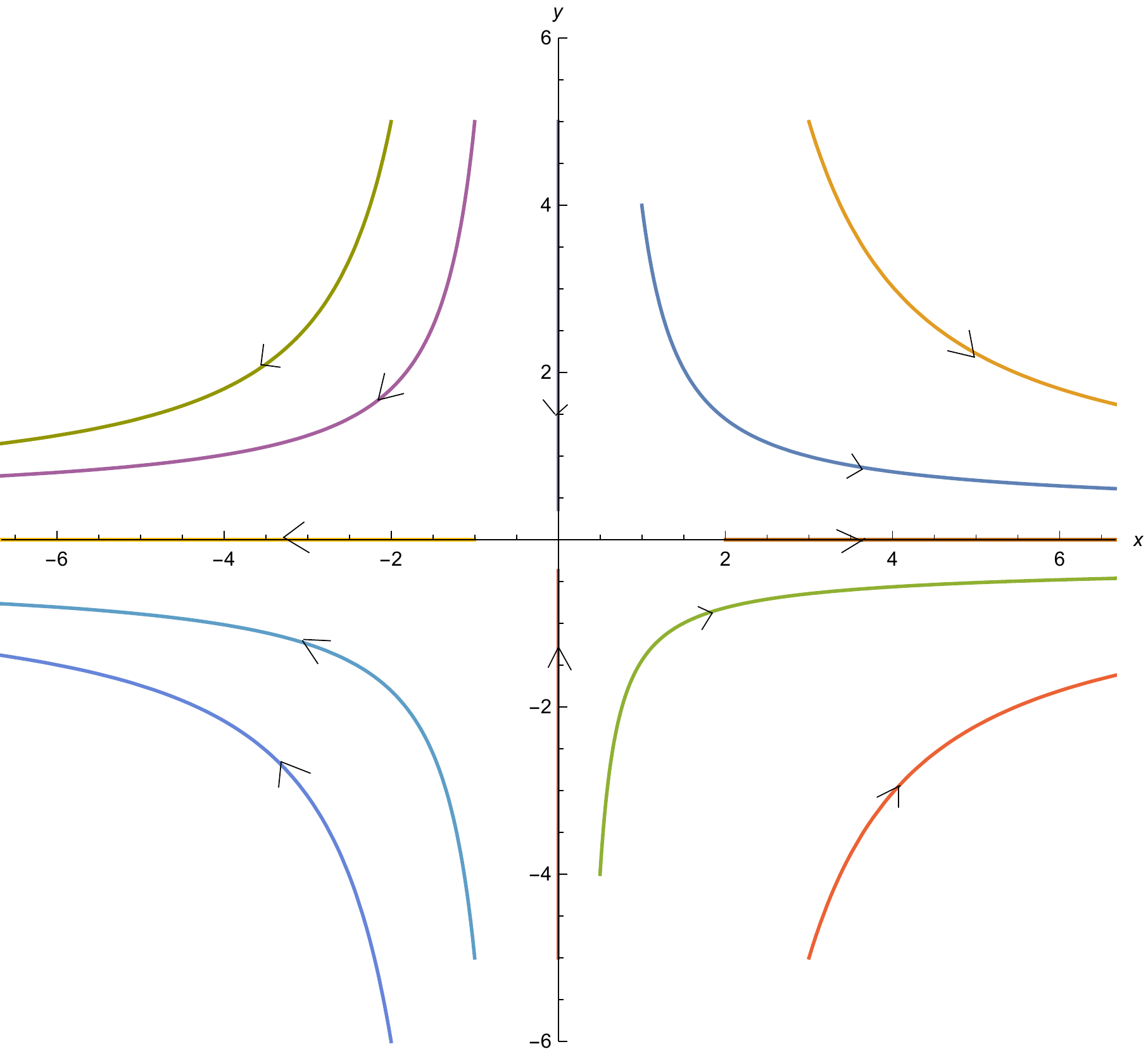}}
        \caption{Saddle\\ $\lambda_1=1$ and $\lambda_2=-2$ }
      \end{figure}
    \\[0.20cm]
    Case(iii): If eigenvalues of $A$ are complex numbers and they satisfy $|arg(\lambda)|<\frac{\alpha\pi}{2}$, then the solutions of system (15) starting in the neighborhood of origin spiral out from origin as shown in Figure (3). The equilibrium in this case is termed as a spiral source \cite{Hirsch}.
    \begin{figure}[h]
      \begin{center}
      \includegraphics[width=0.28\textwidth]{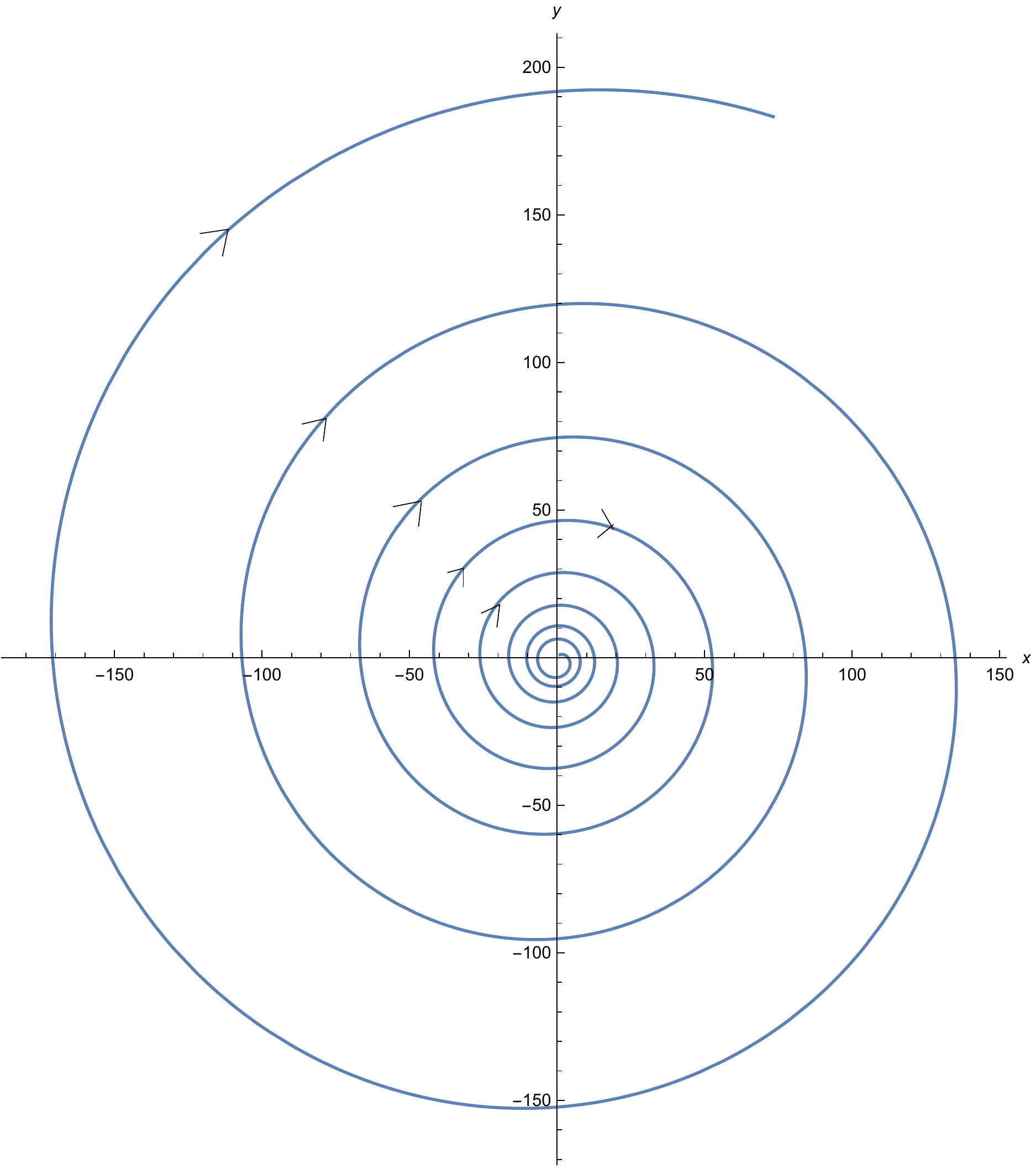}
      \caption{Spiral source \\ $\alpha=0.3$, $\lambda_{1,2}=e^{\pm i\frac{\pi}{7}}$}
      \end{center} 
    \end{figure}
    \\[0.25cm]
      \subsubsection{ Eigenvalues of $A$ are in the stable region:-}
       If $|arg(\lambda)|>\frac{\alpha \pi}{2}$, then  the eigenvalue $\lambda$ is called asymptotically stable. We consider the following two cases:\\[0.20cm]
        Case(i): If 
         $
          A=
          \begin{bmatrix}
          \lambda_1 & 0\\
          0 & \lambda_2
          \end{bmatrix} 
          $
          where $\lambda_1,\lambda_2$ are in $\mathbb{R}$ with $\lambda_1<\lambda_2<0$,
         then according to Theorem \ref{Thm3.1},  all the solutions tend to origin as $t\rightarrow\infty$.\\
        The eigenvalues $\lambda_1$ and $\lambda_2$ are termed as stronger and weaker eigenvalues respectively \cite{Hirsch}.\\
         All the solutions tend to origin, tangentially to the straight-line solution corresponding to the weaker eigenvalue (except those on the strait-line corresponding to the stronger eigenvalue) as shown in the Figure (4).
         \begin{figure}[h]
           \subfloat[$\alpha=0.4$.]{\includegraphics[width=0.35\textwidth]{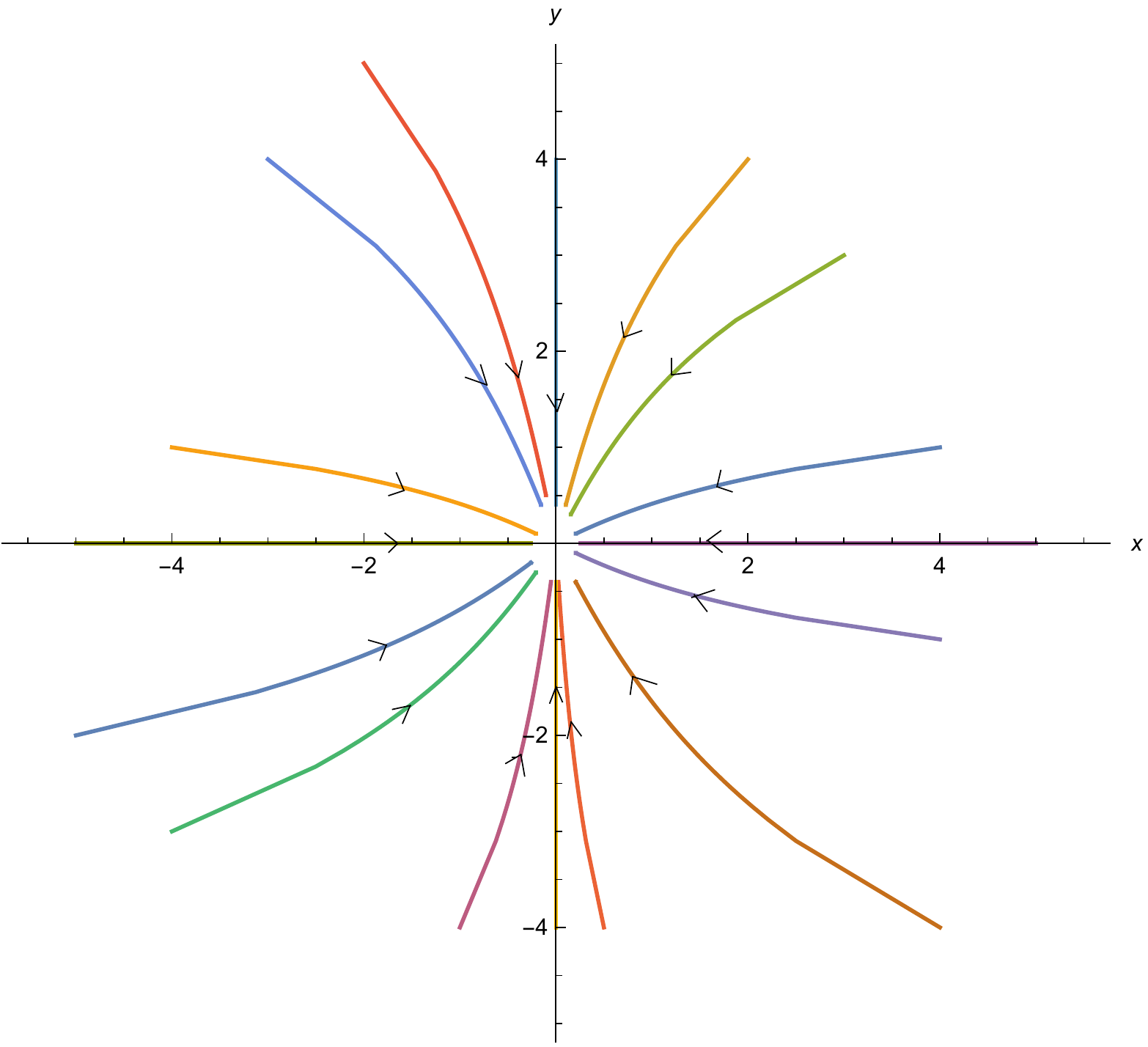}}
           \hfill 
           \subfloat[$\alpha=0.8$]{\includegraphics[width=0.35\textwidth]{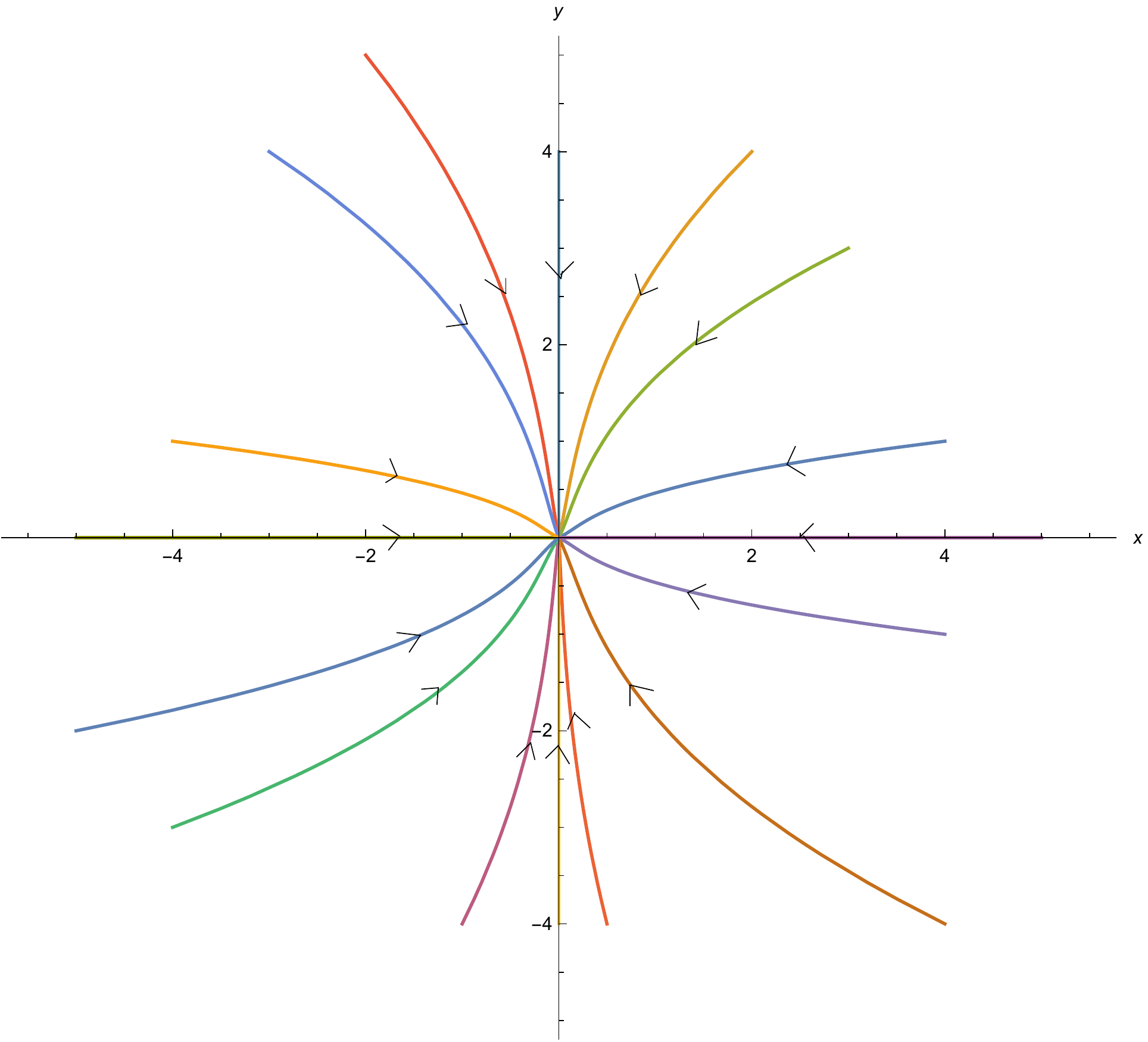}}
           \caption{$\lambda_1=-2$ and $\lambda_2=-1$.}
         \end{figure}
          \\[0.20cm]
         Case(ii): One may obtain a spiral sink as shown in the Figure (5) if $|arg(\lambda)|$ is sufficiently close to $\pi$.\\
         \begin{figure}[h]
        \begin{center} \includegraphics[width=0.4\textwidth]{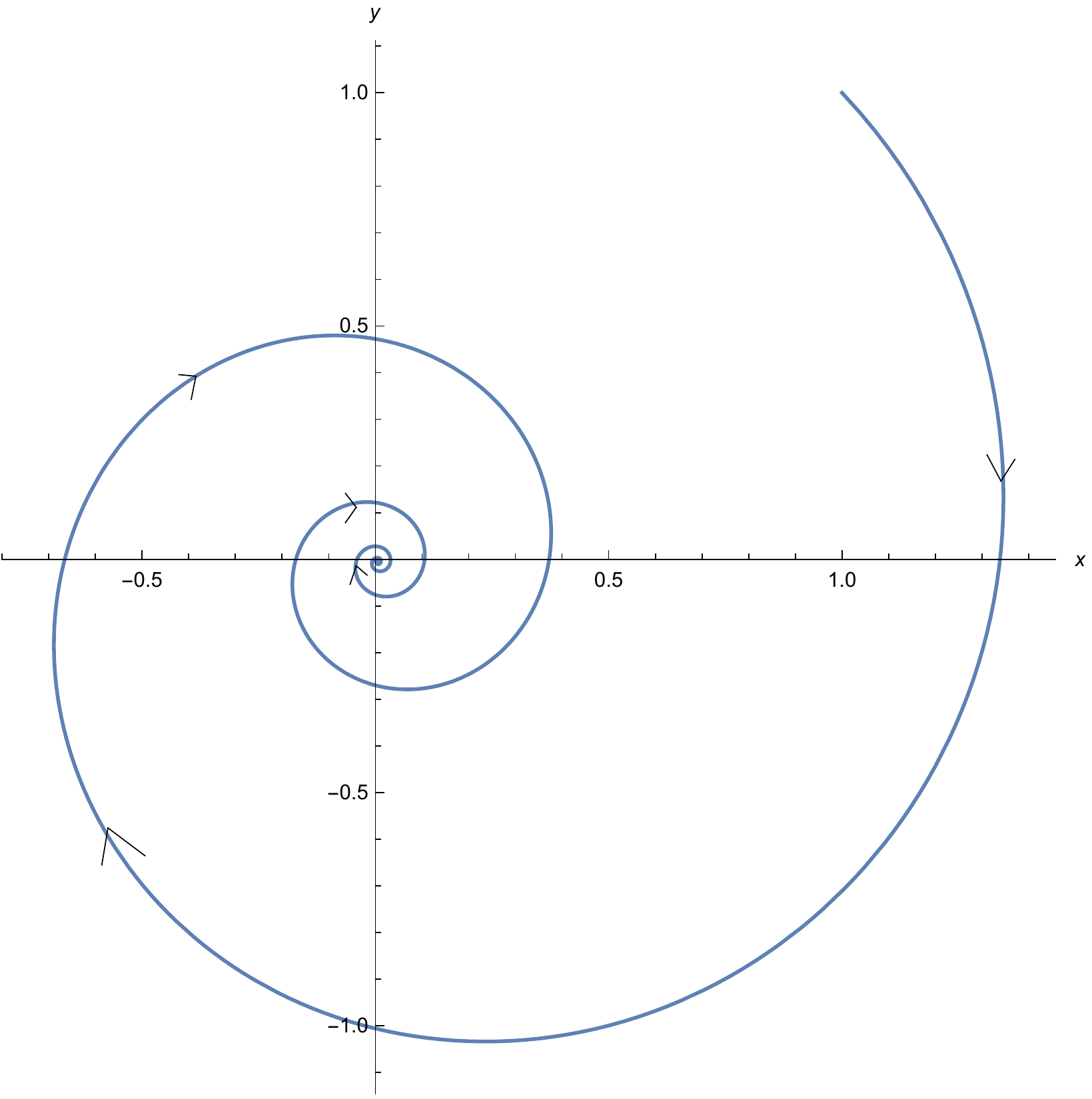}
        \caption{Spiral sink\\$\alpha=0.9$, $|arg(\lambda)|=1.6>\frac{0.9\pi}{2}$}
        \end{center}
         \end{figure}
         Further, if the eigenvalue is very close to boundary of stable region then it may generate a self intersecting trajectory. We discuss this phenomenon in Section \ref{Sec5} in details.
         \section{Eigenvalues at the boundary of stable region}
            If the eigenvalue $\lambda$ of a linear system of ordinary differential equation $$\dot{X}=AX$$ is at the boundary of the stable region. i.e. if $|arg(\lambda)|=\frac{\pi}{2}$ then it will generate a periodic solution. In phase plane, the solution trajectory is a closed orbit.\\
            However, as discussed in \cite{Kaslik}, fractional order system cannot have a periodic solution. In the following Theorem \ref{Thm4.1}, we show that the solutions of such system converge asymptotically to a closed orbit.\\
               \begin{The}\label{Thm4.1}
               The solution trajectory of the system 
               \begin{equation}
               {}_0^C\mathrm{D}_t^\alpha X(t)=AX(t),\quad X(0)=[C_1,C_2]^T, \quad 0<\alpha<1,
               \end{equation}
                where $A$ is $2\times2$ matrix with eigenvalues $\lambda_\pm=re^{\pm i\frac{\alpha \pi}{2}}$ and $C_1, C_2\in \mathbb{R}$,
               converges to a closed orbit, $x(t)^2+y(t)^2=\frac{C_1^2+C_2^2}{\alpha^2}$ in a phase plane. 
               \end{The}  
               \begin{proof}
               The solution of the system (20) is,
               \begin{equation}
     \begin{split}
    X(t) & =
   \begin{bmatrix}
      x(t)\\
      y(t)
  \end{bmatrix}\\
    & =
  \begin{bmatrix}
  Re[E_\alpha(re^{i\frac{\alpha \pi}{2}}t^\alpha)] &  Im[E_\alpha(re^{i\frac{\alpha \pi}{2}}t^\alpha)]\\
  -Im[E_\alpha(re^{i\frac{\alpha \pi}{2}}t^\alpha)] & Re[E_\alpha(re^{i\frac{\alpha \pi}{2}}t^\alpha)]
      \end{bmatrix}
  \begin{bmatrix}
  C_1\\
  C_2
   \end{bmatrix}.
 \end{split}
 \end{equation}
               Thus,
               \begin{align*}
             x(t) & =C_1Re[E_\alpha(re^{i\frac{\alpha \pi}{2}}t^\alpha)]+C_2Im[E_\alpha(re^{i\frac{\alpha \pi}{2}}t^\alpha)] \\
               y(t) & =-C_1Im[E_\alpha(re^{i\frac{\alpha \pi}{2}}t^\alpha)]+C_2 Re[E_\alpha(re^{i\frac{\alpha \pi}{2}}t^\alpha)].
               \end{align*}
               Now,
               \begin{align}
               [x(t)]^2+[ y(t)]^2 & =(C_1^2+C_2^2)\left|E_\alpha(re^{i\frac{\alpha \pi}{2}}t^\alpha)\right|^2. \nonumber \\
               \Rightarrow \lim\limits_{t\to\infty}[[x(t)]^2+[ y(t)]^2] & =(C_1^2+C_2^2)\left|\lim\limits_{t\to\infty}E_\alpha(re^{i\frac{\alpha \pi}{2}}t^\alpha)\right|^2.
               \end{align}
               As $0<\alpha<1$, $|arg(re^{i\frac{\alpha \pi}{2}}t^\alpha)|=\frac{\alpha \pi}{2}$ and for $p=1$ and $\beta=1$ in (14), we have the following asymptotic expansion for Mittag-Leffler function:\\
               \begin{equation*}
               \begin{split}
               E_\alpha(re^{i\frac{\alpha\pi}{2}}t^\alpha) & =\frac{1}{\alpha} exp[r^\frac{1}{\alpha}e^{i\frac{\pi}{2}}t]-\frac{1}{\Gamma(1-\alpha)}\frac{1}{re^{i\frac{\alpha \pi}{2}}t^\alpha}\\
               & \qquad +O\left(\frac{1}{r^2e^{2i\frac{\alpha \pi}{2}}t^{2\alpha}}\right).
               \end{split}
               \end{equation*}
               Since $e^{i(r^\frac{1}{\alpha})t}$ is bounded and the terms $\frac{1}{t^\alpha}$, $\frac{1}{t^{2\alpha}}$, $\frac{1}{t^{3\alpha}}$, $\dots$ tend to zero as $t\rightarrow\infty$.\\
               We have,\\
               $$
               \left|\lim\limits_{t\to\infty}E_\alpha(re^{i\frac{\alpha \pi}{2}}t^\alpha)\right|^2=\frac{1}{\alpha^2}.
               $$
               Therefore,
               \begin{equation}
               \lim\limits_{t\to\infty}[[x(t)]^2+[y(t)]^2]  =\frac{C_1^2+C_2^2}{\alpha^2}.
               \end{equation}
               Hence, the solution trajectory tend to a circle centered at origin and of radius $\frac{\sqrt{C_1^2+C_2^2}}{\alpha}$.\\[0.25cm]
               \end{proof}
               \par In Figure (6), we sketch the trajectories of the case $|arg(\lambda)|=\frac{\alpha\pi}{2}$ for $\alpha=0.01$ and $\alpha=0.7$ respectively. The trajectories are converging to the circles.
               \begin{figure}[h]
               \subfloat[$\alpha=0.01$.]{\includegraphics[width=0.35\textwidth]{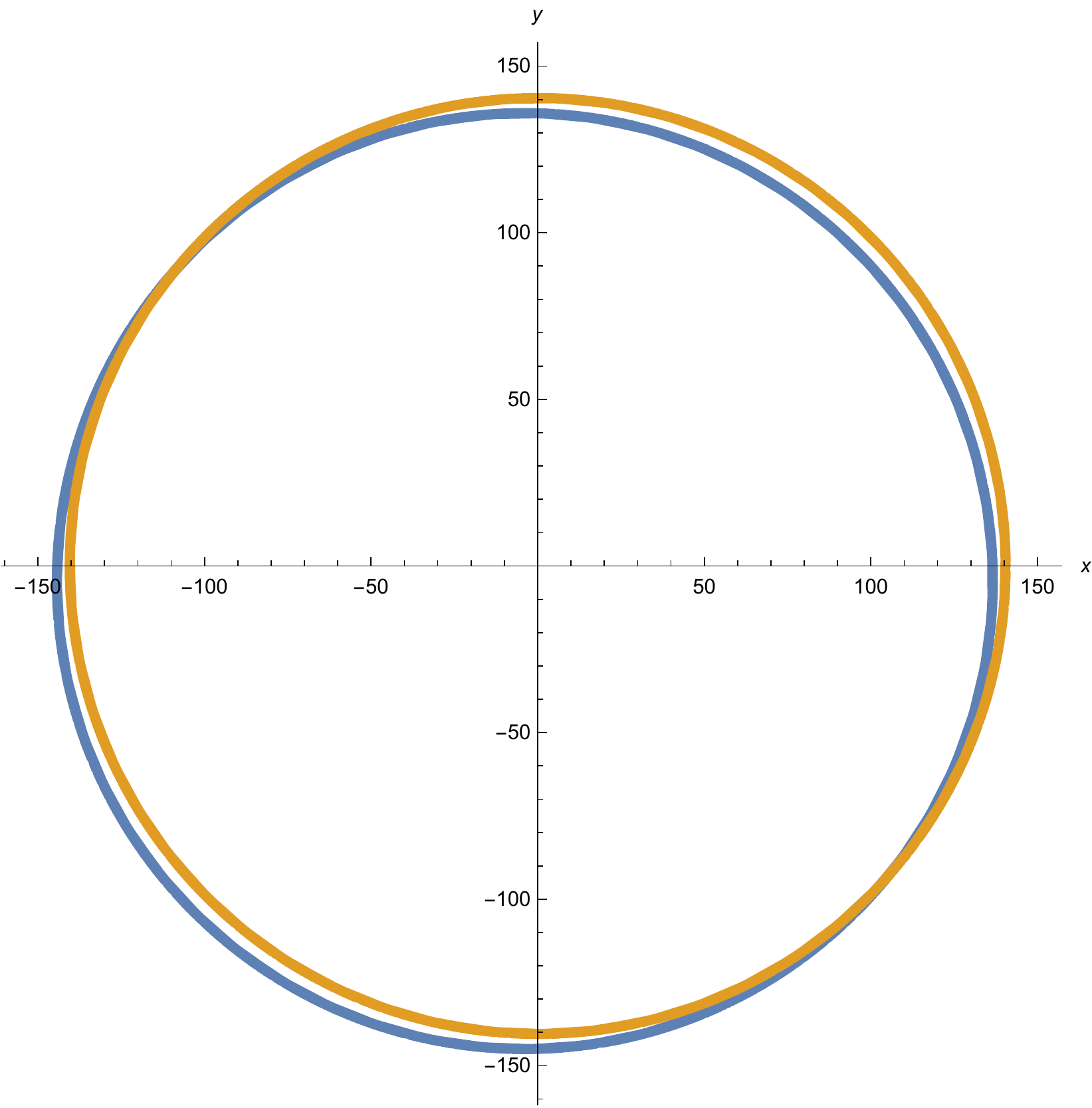}}
              \hfill 
            \subfloat[$\alpha=0.7$]{\includegraphics[width=0.35\textwidth]{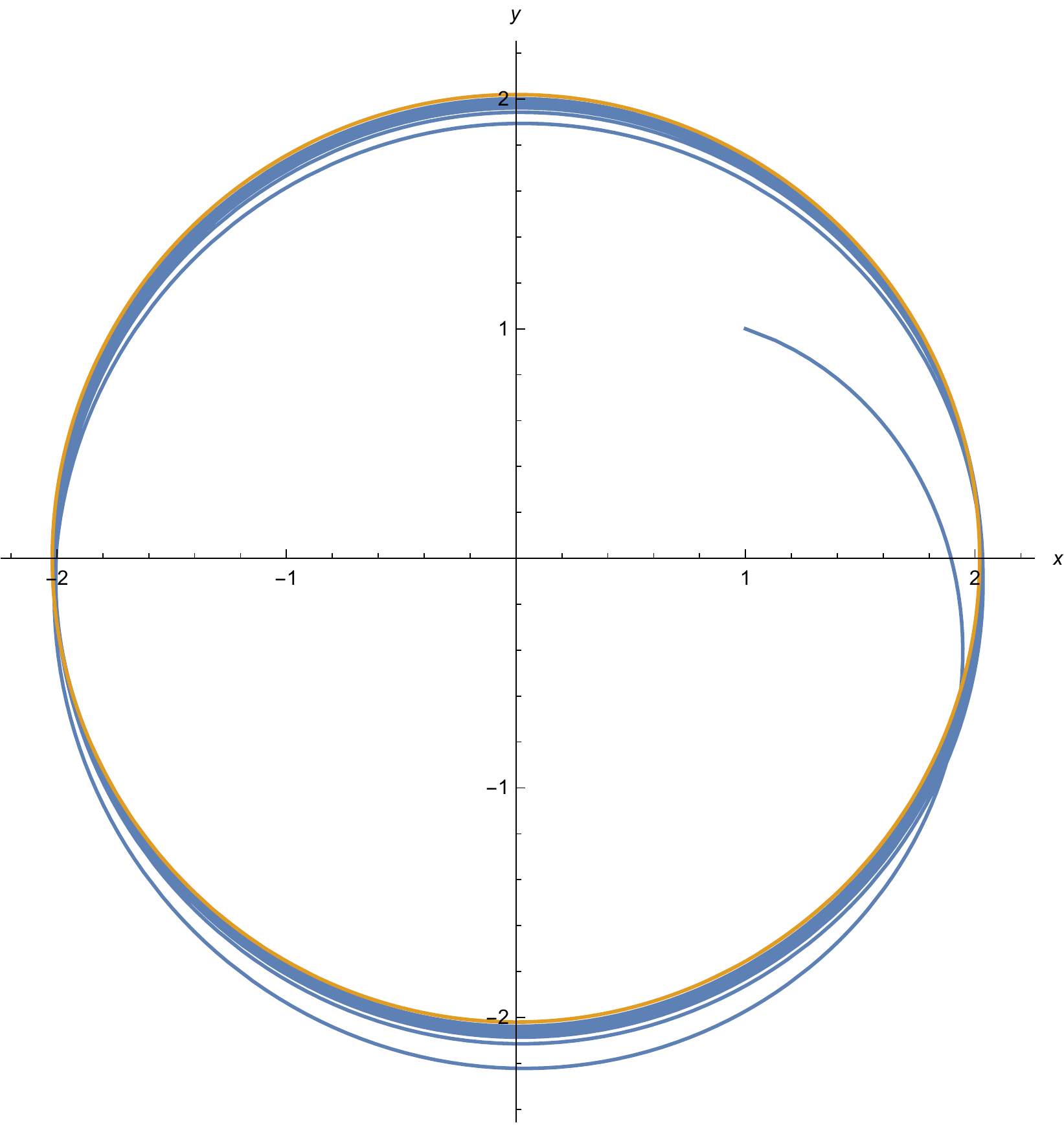}}
             \caption{$\lambda=e^{i\theta}$ where $\theta= arg(\lambda)=\frac{\alpha \pi}{2}$}
                     \end{figure}
         \section{Singular points in solution trajectories}\label{Sec5}
         If the trajectory $X(t)$ is not smooth in the neighborhood of point $X(t_0)$ then it is called a singular point.\\
         Examples of singular points are cusps and multiple points (e.g. Self intersections). \\
         ``A smooth curve has a unique tangent at each point\cite{Pogorelov}" and hence does not contain any singular point.
         \par If the system of ordinary (integer order) differential equations is non-autonomous, then the solution trajectories may intersect.\\
         e.g. In periodically forced pendulum \\
         $\ddot{x}+0.6\sin x=0.3\cos (2\pi t)$
         \, \cite{Mukherjee}, the self intersecting trajectory is given in the Figure (7).
         \begin{figure}[h]
           \begin{center}
           \includegraphics[width=0.5\textwidth]{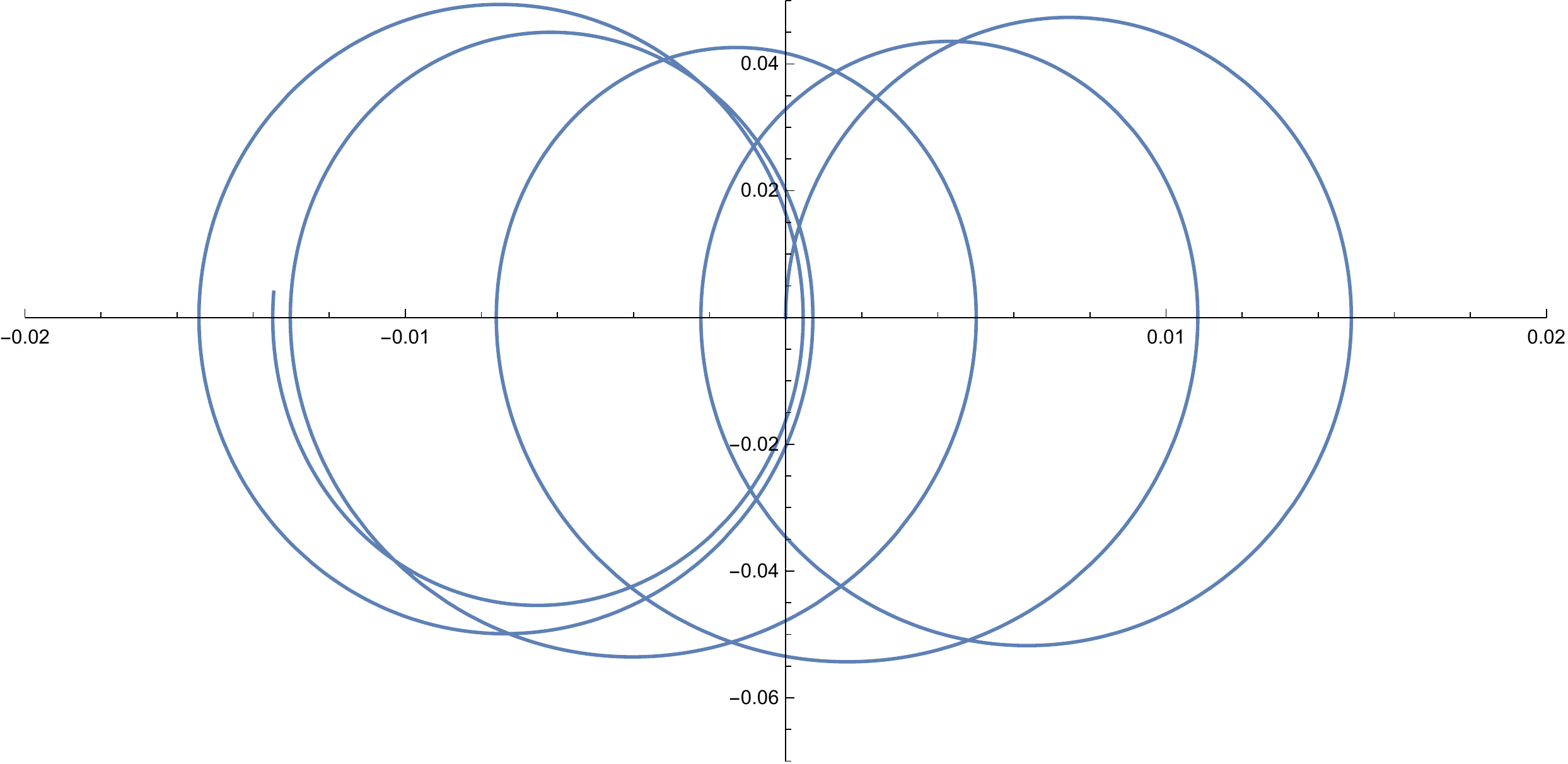}
            \caption{Self-intersecting trajectory of periodically forced pendulum}
           \end{center}               
           \end{figure}
         
         \par Further, it is proved in the literature \cite{Arnold} that the system of autonomous (integer order) differential equations cannot have a self-intersecting trajectory. 
         \par On the other hand, if we consider planar fractional order system then  there may exist singular points in the  trajectories even though the system is autonomous.
         \par We observed self-intersections  and cusps in the solution trajectories of some planar  fractional order systems,
         $${}_0^C\mathrm{D}_t^\alpha X(t)=AX(t).$$
         Based on our observations and the results discussed in Section \ref{Sec3}, we propose the following conjecture:
         \begin{Conj}\label{conj}
         There exist singular points in the  trajectory of planar system ${}_0^C\mathrm{D}_t^\alpha X(t)=AX(t)$ if and only if the eigenvalues $\lambda= re^{\pm i\theta}$ of $A$ satisfy 
         $$
         \frac{\alpha\pi}{2}-\delta_1 < |arg(\lambda)| < \frac{\alpha\pi}{2}+\delta_2, 
         $$
         where $\delta_1>0$ and $\delta_2>0$ are sufficiently small positive real numbers.
         \end{Conj}
         We have stability diagram (cf. Figure (8)) containing three regions viz.:\\
         \begin{figure}[h]
         \begin{center}
        \includegraphics[width=0.5\textwidth]{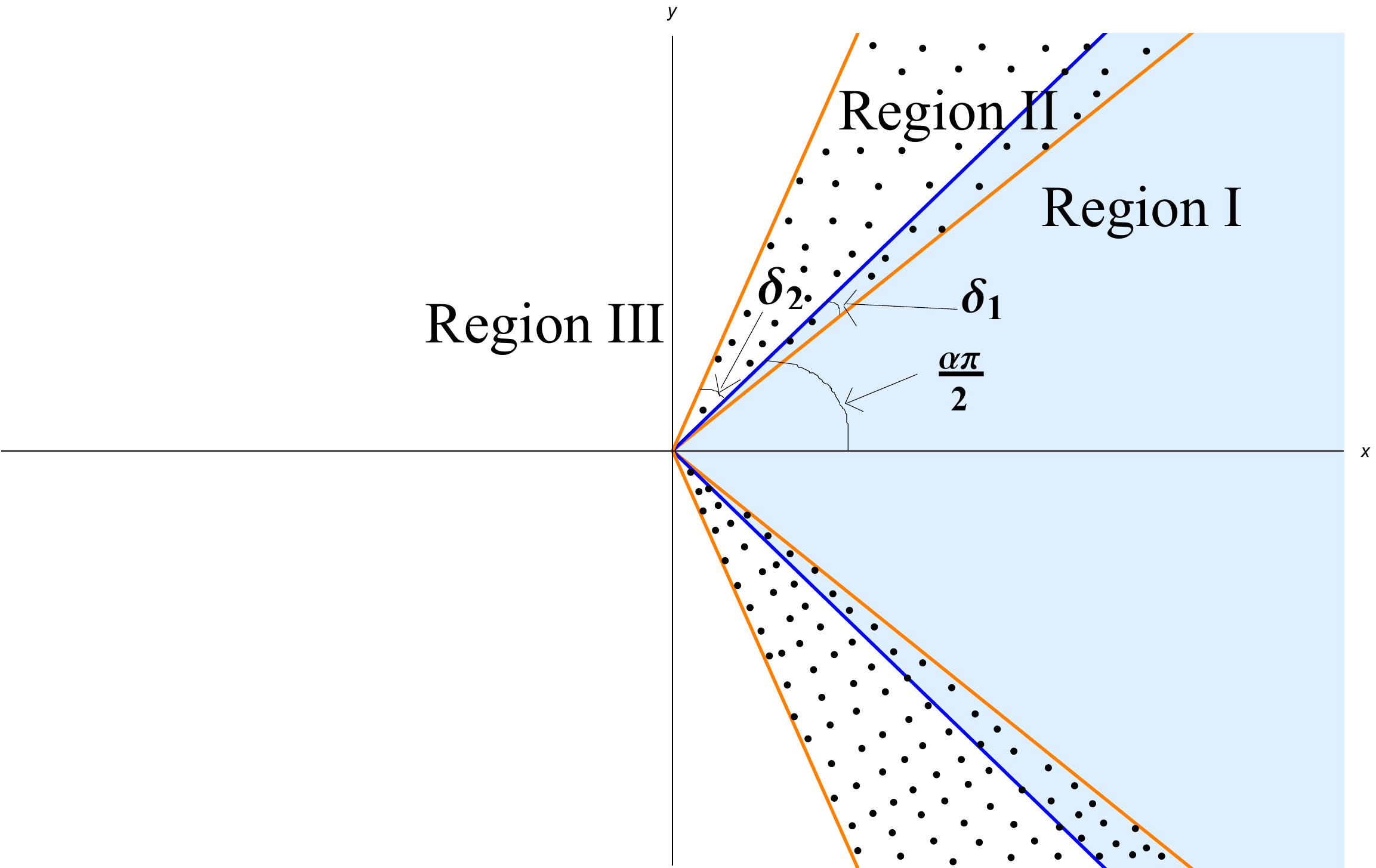}
        \caption{Stability diagram}
        \end{center}               
        \end{figure}
         Region I = \{$\lambda\in\mathbb{C}| 0\le |arg(\lambda)|< \frac{\alpha\pi}{2}$\},\\
         Region II = \{$\lambda\in\mathbb{C}| \frac{\alpha\pi}{2}-\delta_1 < |arg(\lambda)|< \frac{\alpha\pi}{2}+\delta_2$\} and \\
          Region III = \{$\lambda\in\mathbb{C}| \frac{\alpha\pi}{2} \le |arg(\lambda)|< \pi$\}.\\
          Regions I And III are called unstable and stable regions respectively \cite{Tavazoei Haeri}.
          We further observed that $\delta_2>\delta_1$, i.e. Most of the part of Region II is in the stable Region III.\\
          In Table \ref{Tab1}, we list the Region II for some values of $\alpha$ and  eigenvalues $\lambda=re^{\pm i\theta}$ of $A$.\\
            Using Mathematica software, we have verified the existence of self-intersecting trajectories for different values of $\alpha$. Figures 9(a)-9(d) show the singular points in the solution trajectories for $\alpha=0.1$, $\alpha=0.3$, $\alpha=0.6$ and $\alpha=0.9$ respectively. 
            \par We have provided Mathematica code in Appendix 1 so that one can put any value of $\alpha\in(0,1)$ and verify the existence of singular points.
    \begin{figure*}
    \begin{tabular}{c c}
\subfloat[$\alpha=0.1, r=1, \theta=\frac{0.1\pi}{2}+\frac{0.1}{4}$ \qquad \qquad \qquad \qquad \qquad \qquad Cusp and double points.]{\includegraphics[width=0.45\textwidth]{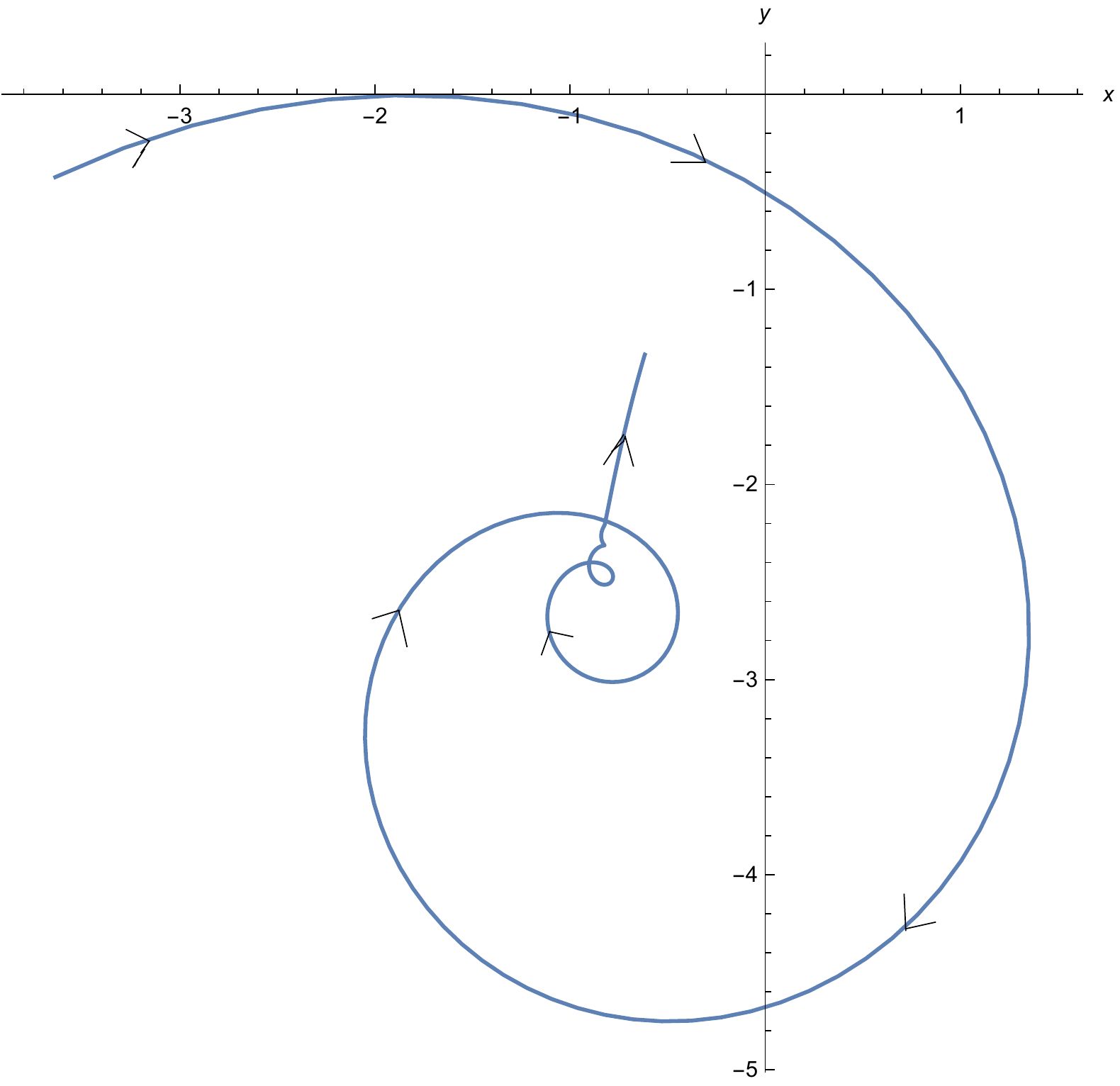}}
        & 
      \subfloat[$\alpha=0.3, r=0.5, 
       \theta=\frac{0.3\pi}{2}+0.029$ \qquad \qquad \qquad \qquad \qquad \qquad  Double and multiple points. ]{\includegraphics[width=0.3\textwidth]{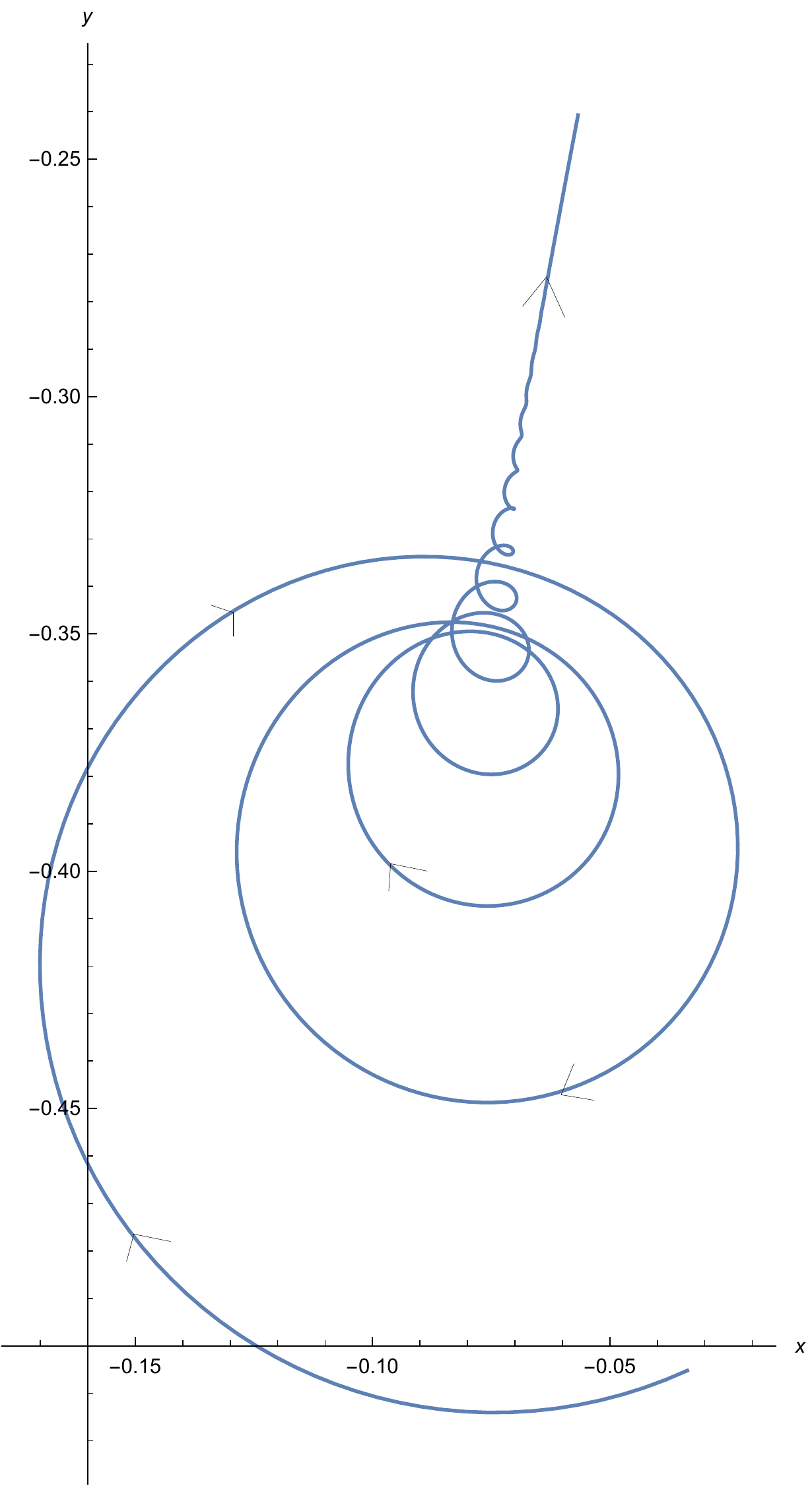}}
      \\
     \subfloat[$\alpha=0.6, r=1, \theta=\frac{0.6\pi}{2}+0.042$ \qquad \qquad \qquad \qquad \qquad \qquad Cusp.]{\includegraphics[width=0.25\textwidth]{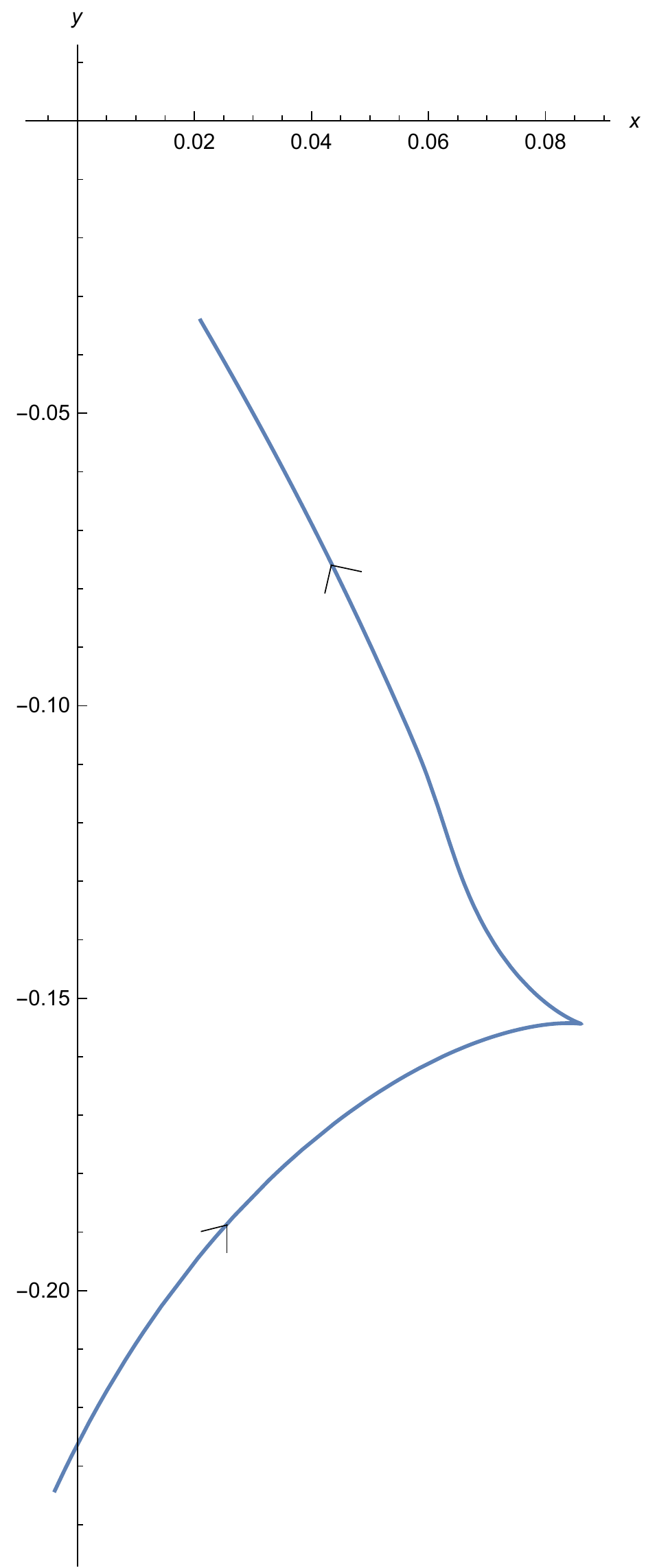}}
    & 
     \subfloat[$\alpha=0.9, r=1.5, \theta=\frac{0.9\pi}{2}+\frac{0.9}{4}$\qquad \qquad \qquad \qquad \qquad \qquad \qquad \qquad \qquad Cusp and double points.]{\includegraphics[width=0.55\textwidth]{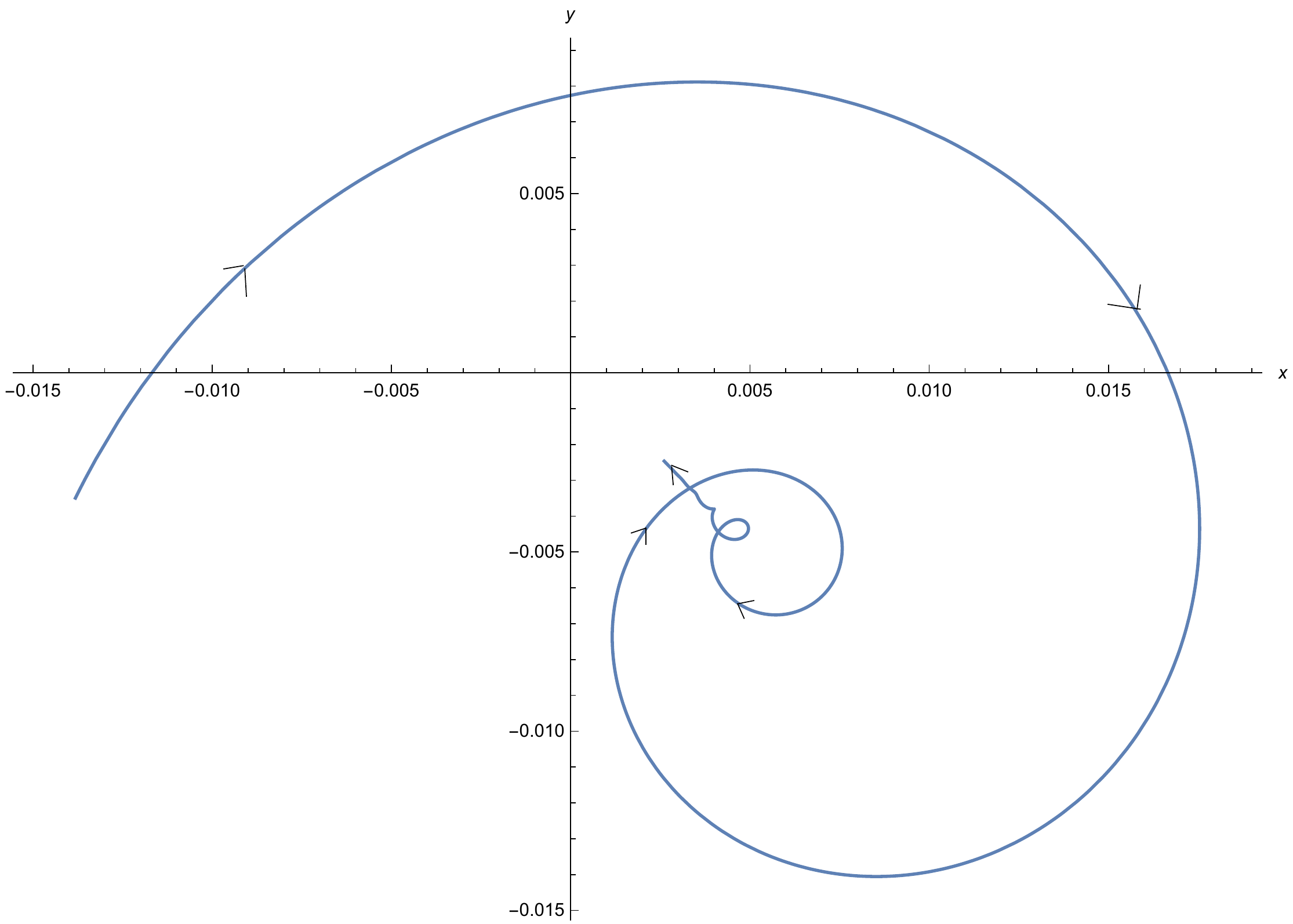}}
    \end{tabular}
    \caption{Singular points in the trajectories of fractional order planar linear systems}
    \end{figure*}
     \begin{table}[h]
    \begin{center}
    \begin{tabular}{|c|c|c|c|c|}
    \hline
    {\bm $\alpha$} & {\bm $\frac{\alpha\pi}{2}$} & {\bm $\delta_1$} & {\bm $\delta_2$} & { Region II} \\
    \hline
   0.1 & 0.15708 & 0.0014 & 0.0639204 & [0.15568, 0.22] \\
    \hline
   0.2 & 0.314159 & 0.0027 & 0.127841 & [0.311459, 0.441] \\
     \hline
   0.3 & 0.471239 & 0.0039 & 0.195761 & [0.467339, 0.665] \\
    \hline
   0.4 & 0.628319 &  0.005 & 0.264681 & [0.623319, 0.891] \\
    \hline
   0.5 & 0.785398 & 0.0057 & 0.341602 & [0.779698, 1.123] \\
    \hline
   0.6 & 0.922478 & 0.0059 & 0.422522 & [0.936578, 1.362] \\
     \hline
   0.7 & 1.09956 & 0.0058 & 0.520443 & [1.09376, 1.613] \\
    \hline
   0.8 & 1.25664 & 0.0049 & 0.633363 & [1.25174, 1.89] \\
   \hline
   0.9 & 1.41372 & 0.0031 & 0.796283 & [1.41062, 2.21] \\
   \hline
    \end{tabular}
    \caption{Values of $\delta_1$ and $\delta_2$ for various $\alpha$}
    \label{Tab1}
    \end{center}
    \end{table}  
    
     {\bf Note:-}\\
  (1) If $\lambda<0$ then the function $E_\alpha(\lambda t^\alpha)$, $0<\alpha<1$ is monotonic \cite{Diethelm}. If $\lambda>0$ then it can be checked that $E_\alpha(\lambda t^\alpha)$, $0<\alpha<1$ is increasing.
 Therefore, if eigenvalues of $A$ are real then there does not exists a self-intersecting trajectory of the system ${}_0^C\mathrm{D}_t^\alpha X(t)=AX(t)$.\\
  (2) Self-intersecting trajectories can also be observed in nonlinear systems \, ${}_0^C\mathrm{D}_t^\alpha X=f(X)$.\\
  According to Hartmann-Grobmann theorem \cite{Li and Ma}, the local behavior of such system will be the same as its linearization ${}_0^C\mathrm{D}_t^\alpha X=J|_{X_*}X$, where $J|_{X_*}$ is the Jacobian of $f$ evaluated at equilibrium $X_*$ of nonlinear system.\\[0.2cm]
  \begin{Ex}
Consider the system \\
  \begin{align}
{}_0^C\mathrm{D}_t^{0.9} x(t) & = x^2-y, \nonumber \\
{}_0^C\mathrm{D}_t^{0.9} y(t) & = x.
  \end{align}
  Here the system has origin as the only equilibrium point.\\
   Therefore, we have
    $
    J|_{(0,0)}=
    \begin{bmatrix}
  0 & -1\\
  1 & 0
    \end{bmatrix}
    $
    and $\lambda_\pm=\pm i$. We can see that $|arg(i)|=\frac{\pi}{2}>\frac{0.9\pi}{2}$.\\
   The trajectory of the system (24) intersects itself, as shown in the Figure (10).
   \begin{figure}
    \includegraphics[width=0.3\textwidth]{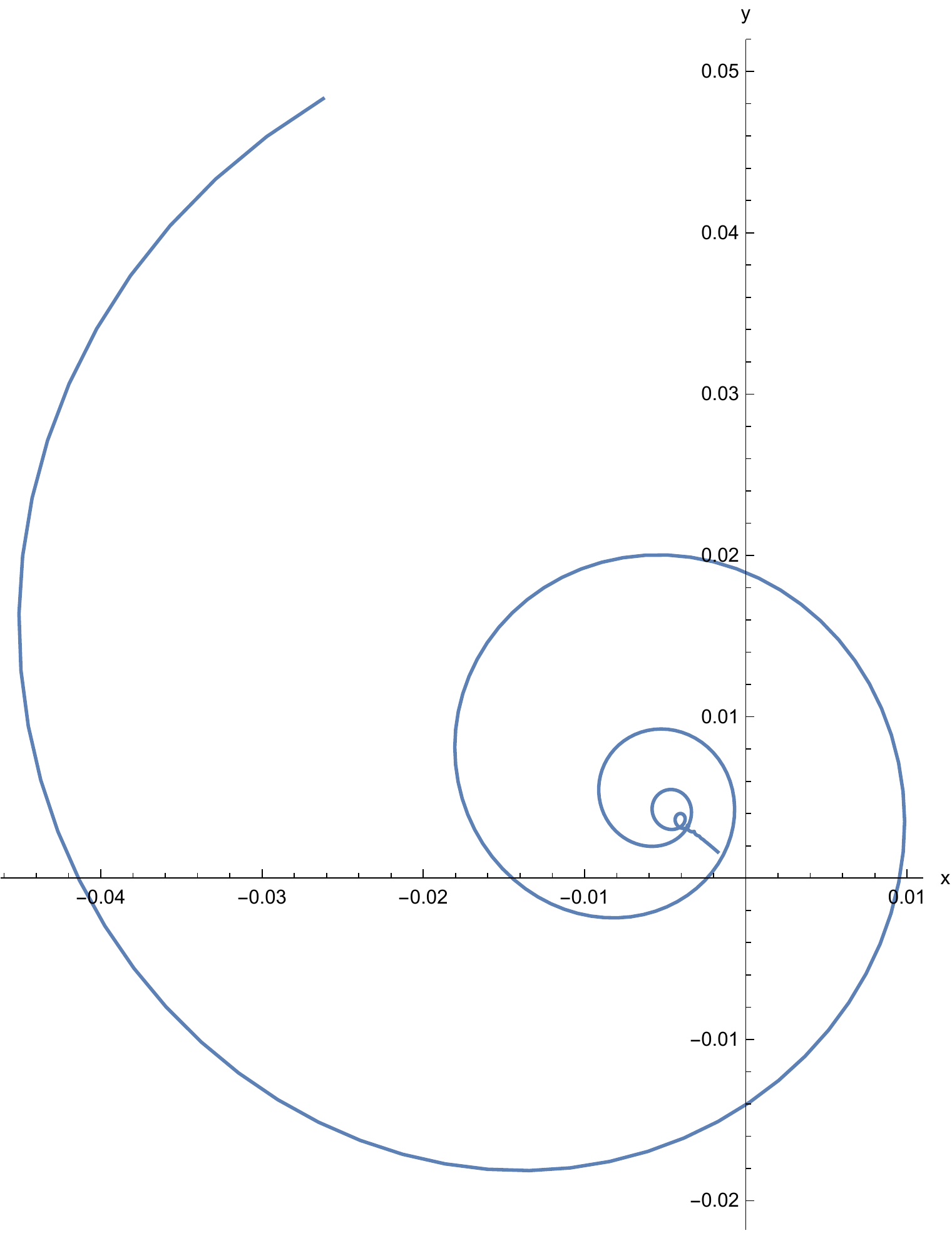}
    \caption{Self-intersecting trajectory of system (24)}            
      \end{figure} 
  \end{Ex} 
  
    \section{Comments on the possible proof of Conjecture \ref{conj}}     
    In this section, we discuss a way to prove  the Conjecture 1.
     \begin{Def}\cite{Karpitschka}
Consider a parametrized curve $X:I\rightarrow\mathbb{R}^2$, where $I\subseteq\mathbb{R}$ is an open interval and $X(t)=(x(t),y(t))$.\\
For $t_0\in I$, a point $X(t_0)$ is said to be a critical point of curve $X$ if 
\begin{equation}
\dot{x}(t_0)=\dot{y}(t_0)=0.
\end{equation}
           \end{Def}
           In most of the examples, the critical points are singular points of the curve. However, the condition (25) is neither necessary nor sufficient for the curve having singular points \cite{Manocha}. (More details are given in Appendix 2).
           \par Consider the linear system ${}_0^C\mathrm{D}_t^\alpha X(t)=AX(t)$. In this case, the solution trajectory is described by equation (17). Without loss of generality, we can set $C_1=1$ and $C_2=0$. Further, assume that the eigenvalues of $A$ are $ r e^{\pm i \theta}$.\\
          $\therefore$ The condition (25) can be written as
          \begin{equation*}
          \frac{d}{dt}
          \begin{bmatrix}
          Re[E_\alpha(re^{i\theta}t^\alpha)]\\
          -Im[E_\alpha(re^{i\theta}t^\alpha)]
          \end{bmatrix}\bigg|_{t=t_0}=
          \begin{bmatrix}
          0\\0
          \end{bmatrix}.
          \end{equation*}
          Equivalently
          \begin{equation}
\frac{d}{dt}E_\alpha(re^{i\theta}t^\alpha)\big|_{t=t_0}=0.
          \end{equation}
          Now,
          \begin{equation*}
          \begin{split}
          \frac{d}{dt}E_\alpha(re^{i\theta}t^\alpha)\big|_{t=t_0}  & = \sum_{k=1}^\infty \frac{\alpha k (re^{i\theta})^k t_0^{\alpha k-1}}{\Gamma{(\alpha k+1)}}\\
           & = \frac{re^{i\theta}}{t_0^{1-\alpha}}\sum_{k=0}^\infty \frac{(re^{i\theta})^k t_0^{\alpha k}}{\Gamma{(\alpha k+\alpha)}}\\
           & = \frac{re^{i\theta}}{t_0^{1-\alpha}} E_{\alpha,\alpha}(re^{i\theta}t_0^\alpha)\, , \qquad t_0>0.
           \end{split}
          \end{equation*} 
          Thus, finding $t_0>0$ satisfying (26) is equivalent to finding zeros of Mittag-Leffler function $E_{\alpha,\alpha}(z)$.\\
          Literature review \cite{Gorenflo, Popsed, Sedletskii} shows that the approximate expressions for ``asymptotic" zeros of Mittag-Leffler function $E_{\alpha,\alpha}(z)$ are known. However, no details are available for the zeros with small absolute values.\\
          In our case, such $t_0$ is usually small.\\
          e.g. If $\alpha=0.2$, $r=1$ then $t_0=18.505$;\\
          $\alpha=0.8$, $r=1$ then $t_0=24.4$ etc.\\
          Thus, to prove the Conjecture 1, one has to prove :\\
          {\bf (a)} $E_{\alpha,\alpha}(re^{i\theta}t_0^\alpha)$ has a zero with sufficiently small absolute value and the unit tangent vector is discontinuous if $\theta\in$ Region II and  $0<\alpha<1$\\
          and \\
          {\bf (b)} If $\theta\notin$ Region II then (as discussed in Appendix 2, the converse of condition (25) is not useful), the map $X$ is injective, proper and regular.
         
  \section{Discussion:}
           The forced damped double-well Duffing equation \cite{Alligood, Chang}
           $$
           \ddot{x}+c\dot{x}-x+x^3=\rho \sin t
           $$
            is an example of nonautonomous planar system exhibiting chaos. 
            There are self-intersecting trajectories of this system in $x\dot{x}$-plane.
            \par We also observed self-intersecting trajectories in (autonomous) planar system of fractional order. The natural question is : ``Can a fractional order autonomous planar system exhibit Chaos?"\\
            As we observed, most of the part of the region II- where the trajectories intersect- lies in the stable region. Since the ``instability of eigenvalues" is a necessary condition for chaos, such stable eigenvalue cannot generate chaotic solutions (See Ex. VII.1).
          \par Further some part of the Region II is in unstable region i.e. there are some unstable eigenvalues leading to self-intersecting trajectories. However, this part of Region II is very small.
          \par It is observed in the literature \cite{v Gejji, Li, C. Li, Wu} that the chaos in integer order system gets disappeared in their fractional order counterparts with sufficiently small values of fractional order $\alpha$.
          \par There is no any reported chaotic fractional order system with system order $<2$.
         \begin{Ex}   
    The equation 
             \begin{equation}
             {}_0^C\mathrm{D}_t^{1.5} x=x(1-x^2)
             \end{equation} 
             discussed in \cite{Edelman} produce only stable eigenvalues as shown below:\\
             The system (27) is equivalent to:
            \begin{align}
            {}_0^C\mathrm{D}_t^{1}x(t) & = y, \nonumber\\
            {}_0^C\mathrm{D}_t^{0.5}y(t) & = x-x^3.
            \end{align}
            The equilibrium points of the system (28) are $E_1=(0,0)$, $E_2=(1,0)$ and $E_3=(-1,0)$. \\
            We discuss the stability of these equilibrium points by using \cite{Tavazoei}. \\
            In this case, $M=LCM(1,2)=2$.\\
            Therefore, $\Delta(\lambda)=\textrm {diag}(\lambda^2,\lambda)-J|_E$, where $J$ is the Jacobian of (28) evaluated at corresponding equilibrium point $E$.\\
            If all roots of $det(\Delta(\lambda))=0$ satisfy $|arg(\lambda)|>\frac{\pi}{2M}$ then the equilibrium point $E$ is asymptotically stable \cite{Tavazoei}. \\
            In this case, roots of $det(\Delta(\lambda))=0$ at $E_1$ are 
            $-1$, $\frac{-1\pm i\sqrt{3}}{2}$ and
            at $E_2$ and $E_3$ are $-2^{1/3}$, $2^{1/3}(\frac{1\pm i\sqrt{3}}{2})$.\\
            Since all these roots are in stable region, equilibrium points of (28) are stable. Thus the system (28) and hence the system (27) cannot generate chaotic solutions.
          \end{Ex}   
            
\section{Conclusion}
         In this article, we have discussed the behavior of the system  ${}_0^C\mathrm{D}_t^\alpha X=AX$. We have considered all the cases of canonical forms of $A$ and provided phase portraits. We have shown that, if eigenvalue $\lambda$ of $A$ satisfy $|arg(\lambda)|=\frac{\alpha\pi}{2}$ (i.e. on the boundary of the stable region), then the trajectory of ${}_0^C\mathrm{D}_t^\alpha X(t)=AX(t)$ tend to a circle $x^2+y^2=\frac{C_1^2+C_2^2}{\alpha^2}$, where $x(0)=C_1$, $y(0)=C_2$. 
         The important observation is the self intersecting trajectories in ${}_0^C\mathrm{D}_t^\alpha X=AX$. We conjectured that there exist singular points in the trajectory if and only if the eigenvalues $\lambda$ of $A$ satisfy 
         $$
         \frac{\alpha\pi}{2}-\delta_1 < |arg(\lambda)| < \frac{\alpha\pi}{2}+\delta_2, 
         $$
         where $\delta_1>0$ and $\delta_2>0$ are sufficiently small positive real numbers. Further, we presented some comments on the possible proof of this conjecture.
        \par We hope that our results will be very useful to the researchers working in this field. The results can be extended to an incommensurate order case as well as to the fractional differential equations involving other types of derivatives.
        
        \section*{Acknowledgments}
        S. Bhalekar acknowledges  the Science and Engineering Research Board (SERB), New Delhi, India for the Research Grant (Ref. MTR/2017/000068) under Mathematical Research Impact Centric Support (MATRICS) Scheme. M. Patil acknowledges Department of Science and Technology (DST), New Delhi, India for INSPIRE Fellowship (Code-IF170439). Authors thank Prof. Andrew Hwang, College of the Holy Cross, Worcester, MA for fruitful discussion on singular points.
   \section*{Appendix 1}
  The Mathematica code to visualize the behavior of the trajectories of ${}_0^C\mathrm{D}_t^\alpha X(t)=AX(t)$ corresponding to the eigenvalues in the region II is given as:
  
  \lstset{
   	tabsize=4,
   	frame=single,
   	language=mathematica,
   	basicstyle=\scriptsize\ttfamily,
   	keywordstyle=\color{black},
   	backgroundcolor=\color{gris245},
   	commentstyle=\color{gray}}
\begin{lstlisting}
Manipulate[
 ParametricPlot[
  Evaluate[{Re[
 MittagLefflerE[alpha, 
 r*Exp[I ((alpha*Pi)/2.0 + epsilon)] t^(alpha)]], -Im[
 MittagLefflerE[alpha, 
 r*Exp[I ((alpha*Pi)/2.0 + epsilon)] t^(alpha)]]}], {t, 0.5
 , b}, AxesLabel -> {x, y}, PlotRange -> All, 
 AxesOrigin -> {0, 0}], {alpha, 0.01, 1,         InputField}, {{r, 1}, 0, 10,
  InputField}, {epsilon, -0.006, 0.8,            InputField}, {b, 10, 500, 
 InputField}]
      \end{lstlisting}
  The output of this code is shown in the Figure (11). In the graphical user interphase generated using the above code, one has to provide the value of fractional order $\alpha$ in first window, values of $r$, $\varepsilon$ and final time `b' in second,   third and fourth window respectively.
  \begin{figure}[h]       \subfloat[$\alpha=0.01$.]{\includegraphics[width=0.34\textwidth]{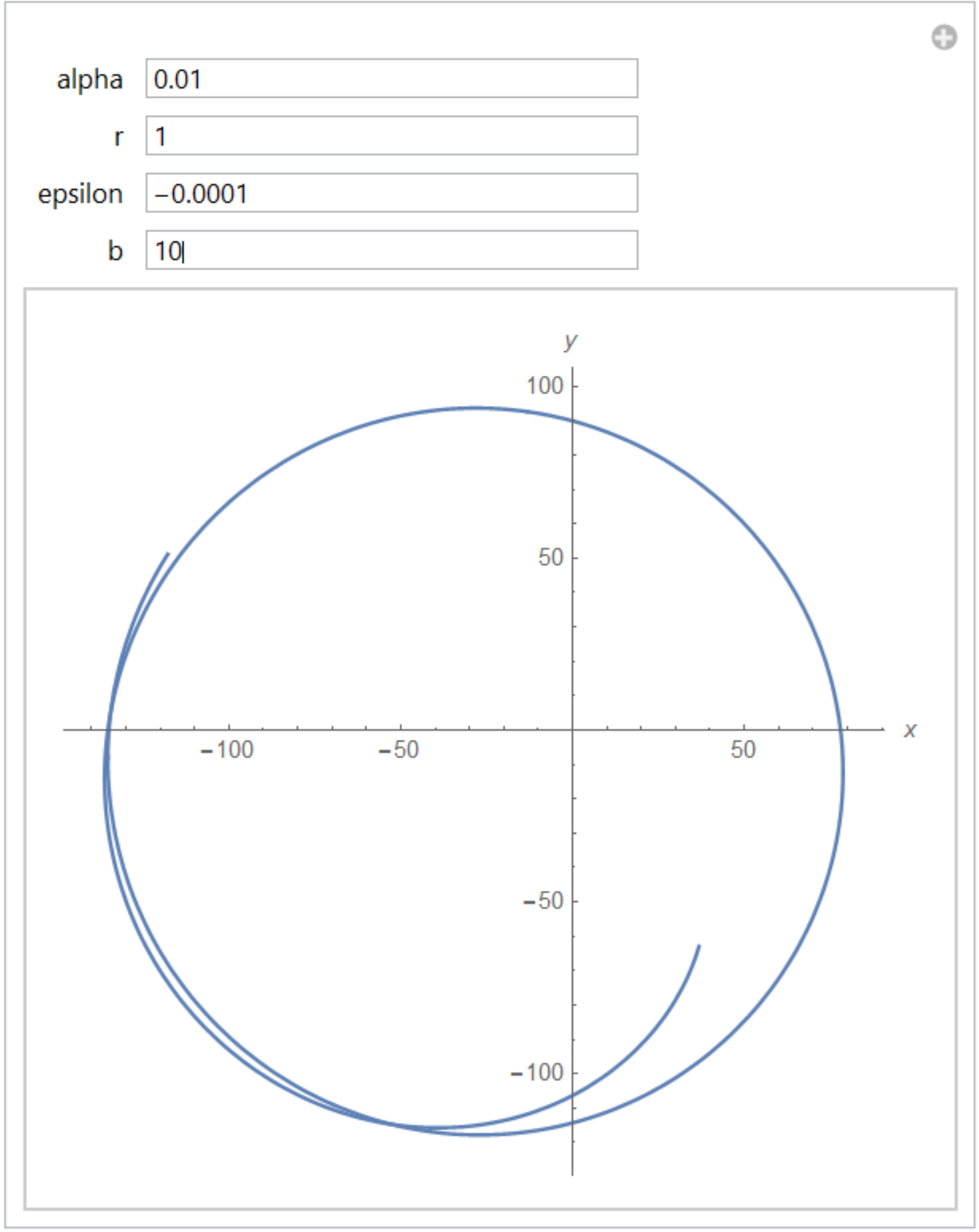}}
     \hfill 
  \subfloat[$\alpha=0.3$]{\includegraphics[width=0.34\textwidth]{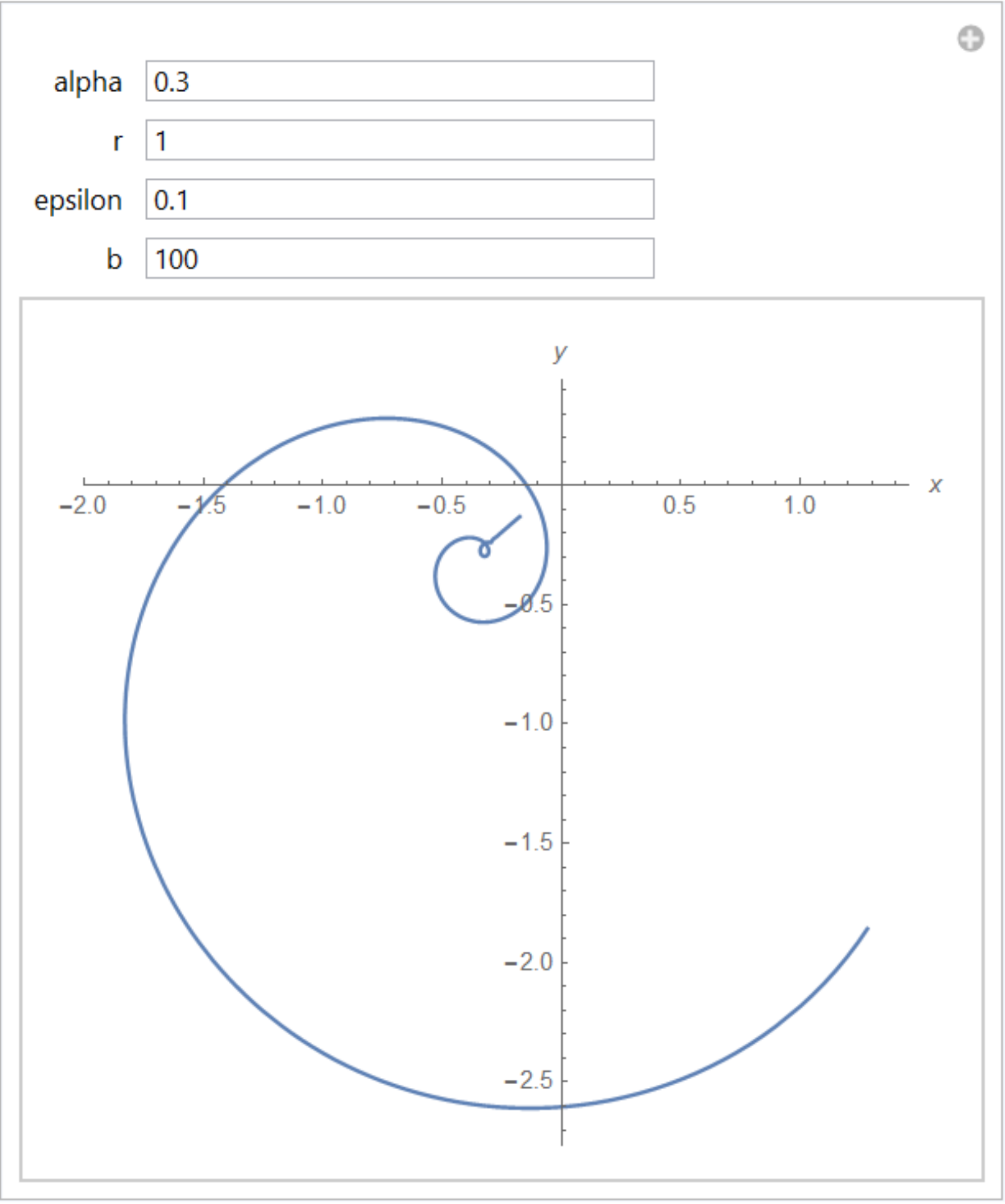}}
  \caption{Graphical user interphase generated using Mathematica}
 \end{figure}
\section*{Appendix 2:\, Singular points of a curve }
Let $I\subseteq\mathbb{R}$ be an open interval.
\begin{Def}\cite{O'Neill}
A mapping $X:I\rightarrow\mathbb{R}^2$ is called regular if $\dot{X}(t)\ne(0,0)$, $\forall t\in I$.
\end{Def} 
\begin{Def}\cite{Guillemin}
A mapping $X:I\rightarrow\mathbb{R}^2$ is called proper if the inverse image of every compact set under $X$ is compact.
\end{Def}
{\bf Remark 1 : Vanishing tangent vector need not imply singular points.}
\begin{The}
Consider a parametric curve $X:I\rightarrow\mathbb{R}^2$ given by $X(t)=(x(t),y(t))$, $t\in I$.
If there exists $t_0\in I$ such that $\dot{x}(t_0)=\dot{y}(t_0)=0$ then either $X(t_0)$ is a singular point or the parametrization is reducible.
\end{The}
\begin{Ex}
Consider $x(t)=t^2$, $y(t)=t^4$ and $X:\mathbb{R}\rightarrow\mathbb{R}^2$ defined by $X(t)=(x(t),y(t))$.\\
We have $\dot{x}(0)=\dot{y}(0)=0$.\\
However, $X$ is a parabola $x^2=y$ and there is no any singular point.\\
It can be checked that the parametrization is reducible. We have $X_1(t)=(t,t^2)$ as a reduced parametrization for the same curve $x^2=y$. \\
Further the unit tangent vector to $X$
$$T(t)=\frac{(2t,4t^3)}{\sqrt{4t^2+16t^6}}=\frac{(1,2t^2)}{\sqrt{1+4t^4}}$$
is continuous on $\mathbb{R}$.
\end{Ex} 
{\bf Remark 2 :} \\
If there is a singular point at $X(t_0)$ then the unit tangent vector $T$ will be discontinuous at $t_0$.\, \cite{Hwang}\\
However the converse is not true. The continuous tangent vector does not imply the non-existence of singular points in the image of $X$.
\begin{Ex}
Consider $X:\mathbb{R}\rightarrow\mathbb{R}^2$ defined by $$X(t)=(t^2,t^3-3t).$$
We have \\
$\dot{X}(t)=(2t,3t^2-3)$  and  
$T(t)=\frac{(2t,3t^2-3)}{\sqrt{9t^4-14t^2+9}}$.\\
The unit tangent $T$ is continuous on $\mathbb{R}$. However, the image of $X$ contains a double point because $X(\sqrt{3})=X(-\sqrt{3})$. i.e. $X$ is not injective.
\end{Ex}  
{\bf Remark 3 :}\\
If $X(I)$ is a smooth curve (i.e. does not have any singular points) then $X$ is injective.\\
The converse is not true.
\begin{Ex}
Consider $X:(-\pi,\pi)\rightarrow\mathbb{R}^2$ defined by $$X(t)=(\sin t,\sin 2t).$$
The image set $X(-\pi,\pi)$ is a figure-8 which contains a singular point.\\
However, $X$ is injective. Note that, $X$ is not a proper.
\end{Ex} 
\begin{The}\cite{Shastri}
If $X:I\rightarrow\mathbb{R}^2$ is regular, proper and injective then $X(I)$ is 1-manifold i.e. a smooth curve.
\end{The}


\begin{thebibliography}{99}
      \bibitem{Das} S. Das, {\it Functional Fractional Calculus}, (Springer Science \& Business Media, Berlin, 2011).
      
      \bibitem{Katugampola} U. N. Katugampola, ``A new approach to generalized fractional derivatives," Bulletin of Mathematical Analysis and Applications,  {\bf 6} 1--15 (2014). 
     
     \bibitem{Atangana} A. Atangana, D. Baleanu, ``New fractional derivatives with non-local and non-singular kernel: Theory and application to heat transfer model," Thermal Science, {\bf 20} 763--769 (2016).
      
      \bibitem{Hristov} J. Hristov, ``Fractional derivative with non-singular kernels: From the seminal definition of Caputo and Fabrizio and beyond with emphasis on diffusion problems," In: Frontiers in Fractional Calculus, S. Bhalekar (Ed.), Bentham Science Publishers, Sharjah, (2018). 
      
      \bibitem{Podlubny} I. Podlubny, {\it Fractional Differential Equations} (Academic Press, New York, 1999).
     
      \bibitem{Oldhalm} K. Oldham, J. Spanier,  {\it The Fractional Calculus: Theory and Applications of Differentiation and Integration to Arbitrary Order} (Elsevier, New York, 111, 1974).
            
     \bibitem{Samko} S. G. Samko, A. A. Kilbas, O. I. Marichev, {\it Fractional Integral and Derivatives: Theory and Applications} (Gordon and Breach Science, Yverdon, 1, 1993).
            
      \bibitem{Kiryakova}  V. S. Kiryakova, {\it Generalized Fractional Calculus and Applications} (CRC press, United states, 1993).
            
      \bibitem{Mathai H} A. M. Mathai, H. J. Haubold, {\it Special Functions for Applied Scientists}  (Springer, New York, 4, 2008).
           
   \bibitem{Mathai Sxena H} A. M. Mathai, R. K. Saxena, H. J. Haubold, {\it The H-Function: Theory and Applications} (Springer Science and Business Media, New York, 2009).
            
    \bibitem{Mainardi} F. Mainardi, {\it Fractional Calculus and Waves in Linear Viscoelasticity: An Introduction to Mathematical Models} (World Scientific, Singapore, 2010).
            
    \bibitem{Magin} R. L. Magin, {\it Fractional Calculus in Bioengineering} (Begell House, Redding,  2006). 
            
    \bibitem{Diethelm Ford} K. Diethelm, N. J. Ford, ``Analysis of fractional differential equations, Journal of Mathematical Analysis and Applications," {\bf 265} 229--248 (2002).
            
  \bibitem{Babakhani} A. Babakhani, V. Daftardar-Gejji, ``Existence of positive solutions of nonlinear fractional differential equations," Journal of Mathematical Analysis and Applications, {\bf 278} 434--442 (2003).
            
   \bibitem{Machado} J. T. Machado, V. Kiryakova,   F. Mainardi, ``Recent history of fractional calculus," Communications in Nonlinear Science and Numerical Simulation, {\bf 16} 1140--1153 (2011).
                    
  \bibitem{Old history} J. T. Machado, V. Kiryakova, F. Mainardi, ``A poster about the old history of fractional calculus," Fractional Calculus and Applied Analysis, {\bf 13} 447--454 (2010).
            
  \bibitem{Recent  history} T. Machado, V. Kiryakova, F. Mainardi, ``A poster about the recent history of fractional calculus," Fractional Calculus and Applied Analysis, {\bf 13} 329--334 (2010).
            
  \bibitem{Matignon1} D. Matignon, ``Stability results for fractional differential equations with applications to control processing," Computational engineering in Systems and Application multiconference, IMACS, lille, france, {\bf 2} 963--968 (1996). 
            
  \bibitem{Matignon2} D. Matignon, ``Stability properties for generalized fractional differential systems, ESAIM proceedings," {\bf 5} 145--158 (1998).
            
   \bibitem{Matignon3} D. Matignon, B. d'Andréa-Novel, ``Some results on controllability and observability of finite-dimensional fractional differential systems," In Computational engineering in systems applications, IMACS, IEEE-SMC Lille, France, {\bf 2} 952--956 (1996). 
            
  \bibitem{Matignon4} D. Matignon, B. d'Andrea-Novel, ``Observer-based controllers for fractional differential systems," In Decision and Control, Proceedings of the 36th IEEE Conference, {\bf 5} 4967--4972 (1997). 
            
   \bibitem{Deng} W. Deng, C. Li, J. L\"{u}, ``Stability analysis of linear fractional differential system with multiple time delays," Nonlinear Dynamics, {\bf 48} 409--416 (2007).
            
    \bibitem{Tavazoei} M. S. Tavazoei, M. Haeri, ``Chaotic attractors in incommensurate fractional order systems," Physica D, {\bf 237} 2628--2637 (2008).
            
   \bibitem{Bhalekar} S. Bhalekar, ``Stability and bifurcation analysis of a generalized scalar delay differential equation," Chaos: An Interdisciplinary Journal of Nonlinear Science, {\bf 26} 084306 (2016).
            
   \bibitem{v Gejji} V. Daftardar-Gejji,  S. Bhalekar, ``Chaos in fractional ordered Liu system, Computers \& mathematics with applications," {\bf 59} 1117--1127 (2010).
            
   \bibitem{Kaslik Shiv} E. Kaslik, S. Sivasundaram, ``Nonlinear dynamics and chaos in fractional-order neural networks, Neural Networks," {\bf 32} 245--256 (2012).
            
   \bibitem{Hartley} T. T. Hartley, C. F. Lorenzo, H. K. Qammer, ``Chaos in a fractional order Chua's system, IEEE Transactions on Circuits and Systems I: Fundamental Theory and Applications," {\bf 42} 485--490 (1995).
            
   \bibitem{C. Li} C. Li, G. Chen, ``Chaos in the fractional order Chen system and its control, Chaos, Solitons \& Fractals," {\bf 22} 549--554 (2004).
            
  \bibitem{Zhang} W. Zhang, S. Zhou, H. Li, H. Zhu, ``Chaos in a fractional-order Rössler system," Chaos, Solitons \& Fractals, {\bf 42} 1684--1691 (2009).
            
  \bibitem{Erdelyi} A. Erdelyi, {\it Higher Transcendental Functions} (McGraw Hill, New York, 3, 1955).
            
  \bibitem{Diethelm} K. Diethelm, {\it The Analysis of Fractional Differential Equations: An Application-Oriented Exposition Using Differential Operators of Caputo Type} (Springer, New York, 2010).
          
  \bibitem{Luchko} Y. Luchko, R. Gorenflo, ``An operational method for solving fractional differential equations with the Caputo derivatives," Acta Math. Vietnam., {\bf 24}  207--233 (1999).
          
  \bibitem{Hirsch} M. Hirsch, S. Smale, R. Devaney, {\it Differential Equations, Dynamical Systems and An Introduction to Chaos} (Academic Press, New York, 2013).
          
  \bibitem{Odibat} Z. M. Odibat, ``Analytic study on linear systems of fractional differential equations," Computers and Mathematics with Applications, {\bf 59}  1171--1183 (2010).
            
  \bibitem{Qian} D. Qian, C. Li, R. P. Agarwal, P. Wong, ``Stability analysis of fractional differential system with Riemann-Liouville derivative," Mathematical and Computer Modelling, {\bf 52} 862--874 (2010).
            
  \bibitem{Kaslik} E. Kaslik, S. Sivasundaram, ``Non-existence of periodic solutions in fractional-order dynamical systems and a remarkable difference between integer and fractional-order derivatives of periodic functions," Nonlinear Analysis, {\bf 13}  1489--1497 (2012).
    
 \bibitem{Pogorelov} A. V. Pogorelov, {\it Differential Geometry} (Noordhoff, Groningen, 1959)
    
 \bibitem{Mukherjee} N. Mukherjee, S. Poria, ``Preliminary Concepts of Dynamical Systems", International Journal of Applied Mathematical Research, {\bf 1} 751--770 (2012).   
    
  \bibitem{Arnold} V.I.Arnol'd, {\it Ordinary Differential Equations} (Springer, New York, 1992). 
           
 \bibitem{Tavazoei Haeri} M. Tavazoei, M. Haeri, ``A note on the stability of fractional order systems", Mathematics and Computers in Simulation, {\bf 79} 1566--1576 (2009).           
         
  \bibitem{Li and Ma} C. Li, Y. Ma, ``Fractional dynamical system and its linearization theorem," Nonlinear Dynamics,  {\bf 71} 621--633 (2013).
   
   \bibitem{Karpitschka} S. Karpitschka, J.  Eggers, A. Pandey, ``Cusp-shaped elastic creases and furrows," Physical review letters,  {\bf 19} 198001 (2017).
          
    \bibitem{Manocha} D. Manocha, ``Regular curves and proper parametrizations," In Proceedings of the international symposium on Symbolic and algebraic computation, 271--276 (1990).       
           
  \bibitem{Popsed} A. Popov, A. Sedletskii, ``Distribution of roots of Mittag-Leffler functions," Journal of Mathematical Sciences, {\bf 190} 209--409 (2013).     
     
  \bibitem{Sedletskii} A. sedletskii, ``Nonasymptotic Properties of Roots
  of a Mittag-Leffler Type Function," Mathematical Notes, {\bf 75}  372--386 (2004).   
           
  \bibitem{Gorenflo} R. Gorenflo, A. Kilbas, F. Mainardi, S. Rogosin, {\it Mittag-Leffler Functions, Related Topics and Applications} (Springer, Berlin, 2, 2014).         
      
   \bibitem{Alligood} K. Alligood, T. Sauer, J. Yorke, {\it Chaos: An Introduction to Dynamical Systems } (Springer, New York, 1996).
           
   \bibitem{Chang} T. Chang, ``Chaotic motion in forced Duffing system subject to linear and nonlinear damping," Mathematical Problems in Engineering, {\bf 2017} 1--8 (2017).
         
  \bibitem{Li} C. Li, G. Chen, ``Chaos and hyperchaos in the fractional-order Rossler equations," Physica A: Statistical Mechanics and its Applications, {\bf 341}, 55--61 (2004).
     
  \bibitem{Wu} X. Wu, S. Shen, ``Chaos in the fractional-order Lorenz system," International Journal of Computer Mathematics, {\bf 86}, 1274--1282 (2009).
    
 \bibitem{Edelman} M. Edelman, ``On the fractional Eulerian numbers and equivalence of maps with long term power-law memory (integral Volterra equations of the second kind) to Grunvald-Letnikov fractional difference (differential) equations," Chaos: An Interdisciplinary Journal of Nonlinear Science, {\bf 25} 073103-1--073103-13 (2015).    
    
  \bibitem{O'Neill} B. O'Neill, {\it Elementary Differential Geometry} (Elsevier, 2006).     
        
    \bibitem{Guillemin} V. Guillemin, A. Pollack, {\it Differential Topology } (American Mathematical Soc., 370, 2010). 
                
     \bibitem{Hwang} Andrew D. Hwang, \url{(https://math.stackexchange.com/users/86418/andrew-d-hwang)}, How to tell whether a curve has a regular parametrization?, \url{https://math.stackexchange.com/q/1961421}
     
     \bibitem{Shastri} A. Shastri, {\it Elements of Differential Topology} (CRC Press, New York, 2011)  
       
            \end{thebibliography}
\end{document}